\documentclass[12pt,letter]{amsart}
\usepackage{amssymb,paralist,xspace,url,amscd,euscript,mathrsfs,stmaryrd,mathtools,
epic,eepic,longtable}
\pdfoutput=1
\usepackage[normalem]{ulem}
\newcommand{\stkout}[1]{\ifmmode\text{\sout{\ensuremath{#1}}}\else\sout{#1}\fi}

\usepackage{enumitem}
\usepackage[centering,text={15.5cm,23cm}]{geometry}
\usepackage[all]{xy}
\SelectTips{cm}{}
\usepackage[colorlinks=true]{hyperref}

\usepackage[dvips,final]{graphicx}
\usepackage{color} 
\usepackage{epsfig}
\usepackage{latexsym}
\usepackage{pstricks}
\usepackage{comment}
	\usepackage{tikz-cd}

\numberwithin{equation}{section}

\allowdisplaybreaks[1]

\newtheorem*{namedtheorem}{\theoremname}
\newcommand{\theoremname}{testing}

\makeatletter
\newtheorem*{rep@theorem}{\rep@title}
\newcommand{\newreptheorem}[2]{%
\newenvironment{rep#1}[1]{%
 \def\rep@title{#2 \ref{##1}}%
 \begin{rep@theorem}}%
 {\end{rep@theorem}}}
\makeatother

\theoremstyle{plain}

\newtheorem{theorem}{Theorem}[section]
\newtheorem{proposition}[theorem]{Proposition}

\newtheorem{proposition-definition}[theorem]{Proposition-Definition}
\newtheorem{lemma-definition}[theorem]{Lemma-Definition}
\newtheorem{corollary}[theorem]{Corollary}
\newtheorem{lemma}[theorem]{Lemma}
\newreptheorem{theorem}{Theorem}
\newreptheorem{lemma}{Lemma}
\newtheorem{construction}[theorem]{Construction}

\theoremstyle{definition}
\newtheorem{definition}[theorem]{Definition}
\newreptheorem{Assumption}{Assumption}
\newreptheorem{definition}{Definition}
\newtheorem{Assumption}[theorem]{Assumption}

\newtheorem{example}[theorem]{Example}

\newtheorem{remark}[theorem]{Remark}

\theoremstyle{remark}

      \newcommand{\R}{{\mathbb R}}

      \newcommand{\M}{{\mathcal M}}
      \newcommand{\MF}{{\mathfrak M}}
      
      \newcommand{\Z}{{\mathbb Z}}
      \newcommand{\K}{{\mathcal K}}
      \newcommand{\X}{{\mathcal X}}
      \newcommand{\A}{{\mathcal A}}
      
      \newcommand{\wt}{\widetilde}
      \newcommand{\etale}{\' etale\xspace}
      \newcommand{\Etale}{\' Etale\xspace}

\newcommand\ul[1]{\underline{#1}}

\newcommand\ir{\operatorname{ir}}

\newcommand\fs{\operatorname{fs}}

\newcommand\spl{\mathrm{spl}}

\newcommand\gl{\mathrm{gl}}

\newcommand\scrM{\mathscr{M}}

\newcommand{\ev}{\mathrm{ev}}

\newcommand\cA{\mathcal{A}}

\newcommand\cK{\mathcal{K}}

\newcommand\cM{\mathcal{M}}
\newcommand\cN{\mathcal{N}}
\newcommand\cO{\mathcal{O}}
\newcommand\cP{\mathcal{P}}

\newcommand\cS{\mathcal{S}}

\newcommand\cX{\mathcal{X}}
\newcommand\cY{\mathcal{Y}}
\newcommand\cZ{\mathcal{Z}}

\newcommand\GG{\mathbb{G}}

\newcommand\kk{\Bbbk}
\newcommand\LL{\mathbb{L}}

\newcommand\NN{\mathbb{N}}

\newcommand\PP{\mathbb{P}}

\newcommand\RR{\mathbb{R}}

\newcommand\ZZ{\mathbb{Z}}

\newcommand\bA{\mathbf{A}}

\newcommand\bP{\mathbf{P}}

\newcommand\bS{\mathbf{S}}

\newcommand\bg{\mathbf{g}}
\newcommand\bp{\mathbf{p}}
\newcommand\bu{\mathbf{u}}

\newcommand{\evt}{\operatorname{evt}}

\newcommand\fC{\mathfrak{C}}

\newcommand\fI{\mathfrak{I}}

\newcommand\fM{\mathfrak{M}}

\newcommand\fV{\mathfrak{V}}

\newcommand\btau{\pmb{\tau}}
\newcommand\bomega{\pmb{\omega}}
\newcommand\bsigma{\pmb{\sigma}}
\newcommand{\brho}{\pmb{\rho}}
\newcommand{\sat}{\operatorname{sat}}

\newcommand\arr{\ifinner\to\else\longrightarrow\fi}

\newcommand\im{\operatorname{im}}

\def\displaytimes_#1{\mathrel{\mathop{\times}\limits_{#1}}}
\def\displayotimes_#1{\mathrel{\mathop{\bigotimes}\limits_{#1}}}

\newcommand\coker{\operatorname{coker}}

\newcommand\Aut{\operatorname{Aut}}
\newcommand\Tor{\operatorname{Tor}}

\newcommand\spec{\operatorname{Spec}}
\newcommand\Spec{\operatorname{Spec}}

\newcommand\rank{\operatorname{rank}}
\newcommand\virt{{\operatorname{virt}}}
\newcommand\id{\mathrm{id}}
\newcommand\pr{\operatorname{pr}}

\newdir{ >}{{}*!/-5pt/@{>}}

\newcommand\doublelong[2]{\mathbin{\xymatrix{{}\ar@<3pt>[r]^{#1}
\ar@<-3pt>[r]_{#2}&}}}

\newlength{\ignora}

\newcommand{\ind}{\operatorname{Ind}}

\renewcommand{\red}{{\mathrm{red}}}

\newcommand{\ol}{\overline}

\newcommand{\gp}{{\mathrm{gp}}}

      \makeatletter
      \def\@setcopyright{}
      \def\serieslogo@{}
      \makeatother
      \makeatletter
      \def\l@subsection{\@tocline{2}{0pt}{2.5pc}{5pc}{}}
      \def\l@subsubsection{\@tocline{2}{0pt}{5pc}{7.5pc}{}}
      \makeatother
   
\newcommand{\changelocaltocdepth}[1]{%
  \addtocontents{toc}{\protect\setcounter{tocdepth}{#1}}%
  \setcounter{tocdepth}{#1}%
}

\usepackage[backend=bibtex,
style=alphabetic,
bibencoding=ascii,
sorting=nyt,
maxalphanames=10,
maxbibnames=99]{biblatex}
\renewbibmacro{in:}{}
\begin{document}

\title{Splitting of Gromov-Witten invariants with toric gluing strata}


\author{Yixian Wu}
\address{Department of Mathematics, The University of Texas at Austin, 2515 Speedway, Austin, TX 78712, USA}
\email{yixianwumath@gmail.com}
\thanks{NSF Grant DMS-1903437}


\date{}

\dedicatory{}


\begin{abstract}
We prove a splitting formula that reconstructs the logarithmic Gromov-Witten invariants of simple normal crossing varieties from the punctured Gromov-Witten invariants of their irreducible components, under the assumption of the gluing strata being toric varieties. The formula is based on the punctured Gromov-Witten theory developed by Abramovich, Chen, Gross and Siebert.
\end{abstract}

\maketitle
\tableofcontents

\section{Introduction}

\changelocaltocdepth{1}
Relative Gromov-Witten invariants of a smooth projective variety $Y$ and a smooth divisor $D$, developed in \cite{AMLYBR2001}\cite{Ionel2000TheSS}\cite{li2001}\cite{li2002}, has been one of the most important techniques to calculate Gromov-Witten invariants. For a degenerating family of projective schemes $X\rightarrow B$ with general fiber over $b\in B$ a smooth variety $X_b$ and the central fiber $X_0$ the union of two smooth irreducible components $Y_1,Y_2$ meeting along a smooth divisor $D$, a degeneration formula is obtained to relate the Gromov-Witten invariants of $X_b$ with the relative Gromov-Witten invariants of $(Y_1,D)$ and $(Y_2,D)$.

Recently, logarithmic Gromov-Witten theory developed in \cite{GSLogGromovWitten} \cite{QileChen2014Slmt} \cite{abramovich2014} has been proved to be a successful generalization of the relative Gromov-Witten theory to the case of $D$ being a normal crossing divisor of $Y$. Especially, for a degenerating family with central fiber $X_0$ a normal crossing variety, a decomposition formula is obtained in \cite{abramovich2020decomposition} that relates the Gromov-Witten invariants of $X_b$ with the logarithmic Gromov-Witten invariants on $X_0$ of rigid decorated \textit{tropical types}. The rigid decorated \textit{tropical types} $\btau$ restrict the combinatorics of the maps, including the dual intersection graphs, the image cones of irreducible components, marked and nodal points, the contact orders and the curve classes. 

To further decompose the logarithmic Gromov-Witten invariants of $X_0$ of type $\btau$ to the invariants of irreducible components of $X_0$, the theory of punctured Gromov-Witten invariants is built in \cite{ACGSPunc}. Punctured Gromov-Witten theory studies logarithmic maps with domain being punctured logarithmic curves, which naturally occur after splitting log smooth curves along nodal points. The combinatorics of the split maps are encoded in tropical sub-types $\btau_1,...,\btau_r$. There is a natural splitting morphism 
\begin{equation*}
\scrM(X/B,\btau) \rightarrow \prod_{i=1}^r \scrM(X/B,\btau_i).
\end{equation*} 

In this paper, we prove an explicit formula (Theorem \ref{main theorem}) presenting the virtual fundamental class of $\scrM(X/B,\btau)$ under splitting as the products of the strata of $\scrM(X/B,\btau_i)$ associated to $\btau_i$-marked tropical types, under the assumption that the gluing strata are toric varieties whose log stratifications are the same as the toric stratifications. A numerical splitting formula of logarithmic Gromov-Witten invariants (Corollary \ref{numericalsplitting}) is obtained as a direct corollary.

\subsection{The main results}
Let $B$ be a log point $(\Spec \kk, \cM_B)$, whose log structure is determined by a chart $Q_B\rightarrow \kk$ with $Q_B$ a toric monoid. Let $X\rightarrow B$ be a projective log smooth morphism between fine, saturated log schemes with Zariski log structures. Let $\beta$ be a curve class in $X$.

The moduli space $\scrM(X/B,\btau)$ of basic stable punctured maps marked by a global decorated type $\btau$ is a logarithmic algebraic stack (\cite[Thm A]{ACGSPunc}). The tropicalization of $\scrM(X/B,\btau)$ is locally determined by the tropical types of the maps. For a geometric point in $\scrM(X/B,\btau)$ with tropical type $\bomega$, there is an \textit{associated basic cone} $\ol{\bomega}$ (Definition \ref{evaluationcone}) of $\bomega$ parametrizing the tropical maps of type $\bomega$. {Supposing $\ol{x}'$ is a geometric point lying in the closure of $\ol{x}$}, there is a canonical \textit{contraction morphism} (Definition \ref{contraction morphism}) from $\bomega'$ to $\bomega$, with $\bomega'$ the tropical type associated to $\ol{x}'$. The contraction morphism induces an inclusion of the associated basic cone $\ol{\bomega}$ as a face of $\ol{\bomega}'$. The tropicalization of $\scrM(X/B,\btau)$ is defined to be the colimit of the basic cones over the geometric points under the above maps. 

In order to define logarithmic evaluation maps, we need to modify the log structure on $\scrM(X/B,\btau)$ based on the set $\bS$ of nodal and punctured points where we evaluate at (Section \ref{logevmapsection}). The tropicalizations of the modified moduli spaces are now determined by the \textit{associated evaluation cones} $\wt{\bomega}_{\bS}$  (Definition \ref{evaluationcone}), parametrizing the tropical maps with type $\bomega$ together with a marking on each edge and leg corresponding to points in $\bS$. 

Splitting a logarithmic map along nodal points of the domain can be described easily using the tropical types. Fix a subset $\bS$ of edges in the graph $G$ of $\btau$. Cutting along each edge $p\in \bS$ results in a set of global decorated types $\btau_1,...,\btau_r$ with $\bS_1,...,\bS_r$ the set of additional half legs from the edges in each type. We use $\wt{\bomega} = \wt{\bomega}_{\bS}$ and $\wt{\bomega}_i = \wt{\bomega}_{i,\bS_i}$ to denote the associated evaluation cones of types $\bomega$ and $\bomega_i$ marked by $\btau$ and $\btau_i$.

\begin{theorem}\cite[Thm C]{ACGSPunc}
There is a finite, representable morphism of moduli spaces of punctured log stable maps to $X$ over $B$
\begin{equation*}\label{splittingmorphism}
    \delta: \scrM(X/B, \btau) \rightarrow \prod_{i=1}^r \scrM(X/B,\btau_i).
\end{equation*}
\end{theorem}

{\color{black}
For each edge $p\in\bS$, the tropical type $\btau$ determines a cone $\bsigma(p)$ of $\Sigma(X)$ (Definition \ref{globaltypedefinition}), and a log scheme $V_p: = V_X(\bsigma(p))$, the logarithmic stratum of $X$ of $\bsigma(p)$. The logarithmic subscheme $V_p$ is the gluing stratum where the nodal point of $p$ is restricted on by $\btau$. The reverse process of gluing punctured maps of type $\btau_i$ requires both schematic and tropical matching for nodal points. Though in general complicated,  under the case of the gluing strata being toric varieties, the gluings of the logarithmic maps are completely determined by the tropical information.

\begin{Assumption}
\label{toricassumption}
Assume $B = \spec(\kk\rightarrow Q_B)$ is  a log point with $Q_B$ a toric monoid. Suppose $X\rightarrow B$ is an integral, log smooth morphism between fine, saturated log schemes. Assume $\ol{\cM}_X$ is globally generated, and for each edge $p\in\bS$, the strict closed subscheme $V_p$ of the log scheme $X$ has the underlying scheme a toric variety, and the log stratification of $V_p$ is the same as the toric stratification.  
\end{Assumption}
}
\begin{lemma}(Proposition \ref{Vp proposition} + Theorem \ref{toricstratacondition})
There exists a toric variety $X_p$ associated to the fan $(\Sigma_p,N_p)$ with canonical toric log structure, such that $V_p$ is isomorphic to a toric stratum of $X_p$.
\end{lemma}

Following the construction of $\Sigma_p$, there is a map of tropicalization $\Sigma_p \rightarrow \Sigma(B)$. Let $\Sigma_{\bomega}^B$ be the relative product of $\Sigma_p$ over the tropicalization of $B$, and 
$\prod_{p\in \bS}^B N_{p,\RR}$
 be the relative product of $N_{p,\RR}$ over $N_{B,\RR}$. The logarithmic evaluations along markings induce the tropical evaluations  maps
\begin{equation}
\evt_{\bomega}^B :\wt{\bomega}_{\bS} \rightarrow \Sigma_{\bomega}^B \rightarrow \prod_{p\in \bS}^B N_{p,\RR}
.
\end{equation}
For curves of types $\bomega_i$, the tropical matching condition is a fiber diagram
\begin{equation*}
\begin{tikzcd}
(\varepsilon_{\bomega}^B)^{-1}(0) \arrow[r] \arrow[d] &  \prod_i^B \wt{\bomega}_{i}
\arrow[d,"\varepsilon_{\bomega}^B"] \\
0 \arrow[r] & \prod_{p\in \bS}^B N_{p,\RR},
\end{tikzcd}
\end{equation*}
where 
\begin{equation} \label{varepsilondefinitionoverB}
\varepsilon_{\bomega}^B: \prod_i^B \wt{\bomega}_{i} \xrightarrow{\prod \evt^B_{\bomega_i}} \prod_{p\in \bS}^B N_{p,\RR} \times \prod_{p\in \bS}^B N_{p,\RR} \xrightarrow{\prod \coker \ol{\Delta}_p} \prod_{p\in \bS}^B N_{p,\RR},
\end{equation}
with the second map the cokernel of the diagonal map. A lot of times, it is more convenient to work with the absolute products than the relative products over $B$. So we introduce the tropical evaluation maps 
\begin{equation}
\evt_{\bomega} :\wt{\bomega}_{\bS} \rightarrow \Sigma_{\bomega} \rightarrow \prod_{p\in\bS} \Sigma_p \rightarrow \prod_{p\in \bS} N_{p,\RR}.
\end{equation}
and 
\begin{equation} \label{varepsilondefinition}
\varepsilon_{\bomega}: \prod_{i=1}^r \wt{\bomega}_{i} \xrightarrow{\prod \evt_{\bomega_i}} \prod_{p\in \bS} N_{p,\RR} \times \prod_{p\in \bS} N_{p,\RR} \xrightarrow{\prod \coker \ol{\Delta}_p} \prod_{p\in \bS} N_{p,\RR}
\end{equation}
with all products being absolute products. For the rest of the paper, we mostly work with the absolute products. We will emphasize the relative products when they are used.

The maps \eqref{varepsilondefinitionoverB} and \eqref{varepsilondefinition} tell the difference by evaluating at two half edges after splitting. Instead of requiring the evaluations to be matched along split edges, we introduce \textit{generic displacement vectors} and require the maps to be matched after the perturbation along this vector. The minimal types satisfying the new matching conditions determine the components of a substack of $\prod_{i=1}^r \scrM(X/B,\btau_i)$ rationally equivalent to $\delta(\scrM(X/B,\btau))$. The idea is inspired by the intersection theory of toric varieties in \cite{FS97}. 

\begin{definition} \label{globalgluingdata}

\begin{enumerate}
\item A vector $\fV$ in $\prod_{p\in \bS}N_{p}$ is a \textit{displacement vector} if it lies in the sublattice generated by $\prod_{p\in \bS}^B N_{p}$, the relative product over $N_B$.
\item A displacement vector $\fV$ is a \textit{generic displacement vector} if for all types $[\brho] = (\brho_1,...,\brho_r)$ such that 
\begin{enumerate}[label = (\roman*)]
\item $\brho_i$ admits a contraction morphism to $\btau_i$, for $i=1,...,n$,
\item $\fV \in \im(\varepsilon_{\brho})$, for $\varepsilon_{\brho}$ defined in \eqref{varepsilondefinition}, and 
\item 
\begin{equation*}
\begin{split}
\sum_{i=1}^r \dim \wt{\brho}_i- \dim \wt{\btau} & = \sum_{p\in \bS}\dim N_p - (|\bS|-r+1) \cdot \rank Q_B^{\gp}.
\end{split}
\end{equation*}
\end{enumerate}
then, the maps $[\varepsilon_{\brho}^B]^{\gp}$ are surjective, for the relative map $\varepsilon_{\brho}^B$ defined in \eqref{varepsilondefinitionoverB}.

We use $\Delta(\fV)$ to denote the set of types $[\brho]$ satisfying the above conditions. We will show in Remark \ref{equivalentRemark} that for all $[\brho]\in \Delta(\fV)$, $\fV$ does not lie in the boundary of $\im(\varepsilon_{\brho})$, which is the definition used in the previous version of this paper.

By condition (i), the types in $\Delta(\fV)$ determine strata in $\prod_{i=1}^r \scrM(X/B,\btau_i)$. Condition (ii) requires the existence of tropical maps with type $(\brho_1,\ldots, \brho_r)$ that match along splitting edges after the perturbation along $\fV$.  Condition (iii) requires the types to have expected virtual dimension. Note that the right hand side of condition (iii) is the dimension of the relative product $\prod_{p\in\bS}^B N_{p,\RR}$. Since the image of $\wt{\btau}$ under $\varepsilon^B_{\btau}$ is $0$, the type $[\brho]$ has the minimal dimension that allows $[\varepsilon_{\brho}^B]^{\gp}$ to be surjective.

\item For each $[\brho]\in \Delta(\fV)$, we define the multiplicity
\begin{equation*}
    m_{[\brho]} = \big[\im(\ol{\varepsilon}_{\brho})^{\sat}:\im(\ol{\varepsilon}_{\brho})],
\end{equation*}
where $\ol{\varepsilon}_{\brho_i}$ is the lattice map associated to ${\varepsilon}_{\brho_i}$ and $\im(\ol{\varepsilon}_{\brho})^{\sat}$ is the saturation of the sublattice $\im(\ol{\varepsilon}_{\brho})$ in $\prod_{p\in \bS}N_p$. 

\end{enumerate}
\end{definition}

\begin{remark} \label{equivalentRemark}
\begin{enumerate}
\item A generic displacement vector $\fV$ must lie in the interior of $\im(\varepsilon_{\brho})$ for all $[\brho] \in \Delta(\fV)$. Suppose not, then $\fV$ lies in $\im(\varepsilon_{\brho'})$ where $[\brho']$ is a strict contraction of $[\brho]$. Therefore, the sum $\sum_{i=1}^r \dim \wt{\brho}'_i$ does not satisfy condition (iii) as it is too small. By adding nodes on edges or legs of $[\brho']$, we can enlarge the dimension and construct a type $[\brho'']$ satisfying the conditions in Definition \ref{globalgluingdata}. Thus, the type $[\brho'']$ is in $\Delta(\fV)$. However, the map $[\varepsilon_{\brho''}^B]^{\gp}$ is equal to $[\varepsilon_{\brho'}^B]^{\gp}$, which can not be surjective, as the dimension of $\wt{\brho_i}'$ is not large enough. It contradicts to the assumption that $\fV$ is generic.
\item When the dual intersection graph $G$ is a tree, the dimension condition becomes 
\begin{equation*}
\sum_{i=1}^r \dim \wt{\brho}_i- \dim \wt{\btau} = \sum_{p\in \bS}\dim N_p.
\end{equation*}
Suppose the logarithmic point $B$ is trivial, then the definition of generic displacement vector here can be rephrased as the definition of \textit{general displacement vector} in \cite[Def.A.2]{canonical_wall}, where a general displacement vector is a vector $\fV$ satisfying the condition that for all types $\brho_i$ with $\im(\varepsilon_{\brho}^{\gp})$ not being surjective, $\fV$ does not lie in $\im(\varepsilon_{\brho}^{\gp})$.
\end{enumerate}
\end{remark}

Now, we are ready to state the main result:

\begin{theorem} \label{main theorem}
Let $X$ be a fine, saturated logarithmic projective scheme, log smooth over a log point $B = \spec(Q_B\rightarrow \kk)$, with $Q_B$ a toric monoid. Let $\btau$ be a decorated global tropical type. Fix a set of the splitting edges $\bS$ and let $\btau_1,\ldots\btau_r$ be the decorated global types obtained after splitting.

Suppose the Assumption \ref{toricassumption} is satisfied. Let $\fV$ be a generic displacement vector defined in Definition \ref{globalgluingdata}. Then, for the finite, representable morphism of moduli spaces of punctured stable log maps
\begin{equation*}
    \delta: \scrM(X/B, \btau) \rightarrow \prod_{i=1}^r \scrM(X/B,\btau_i),
\end{equation*}
the following equation holds
\begin{equation} \label{main equation}
    \delta_*[\scrM(X/B, \btau)]^{\virt} = \sum_{[\brho] \in \Delta(\fV)}
    \prod_{i=1}^r \frac{m_{[\brho]}}{|\Aut(\brho_i/\btau_i)|} \cdot  j_{\brho_i\btau_i*}[\scrM(X/B,\brho_i)]^{\virt},
\end{equation}
with $j_{\brho_i\btau_i}$ the finite morphism from $\scrM(X/B,\brho_i)$ to $\scrM(X/B, \btau_i)$ associated to the contraction morphism $\brho_i\rightarrow \btau_i$, and $\Aut(\brho_i/\btau_i)$ the automorphism group of $\brho_i$ relative to $\btau_i$.
\end{theorem}

A special case of Theorem \ref{main theorem} is the splitting of $\btau$ at all edges. Then each
split type $\btau_i$ consists of one vertex with a number of legs, with the associated image stratum strictly smaller than the
full target. In this case, \eqref{main equation}
expresses the punctured invariants of type $\btau$ in terms of punctured
invariants of these logarithmic strata. For example, in a degeneration situation
as in \cite{abramovich2020decomposition}, this expresses the Gromov-Witten invariants of a
general fiber in terms of the punctured invariants of the strata of the central
fiber. Such localization to the strata does not follow from the general gluing formulas in \cite{parker2017tropical}, \cite[Thm C]{ACGSPunc} and \cite{ranganathan2020logarithmic}.

A direct corollary of the above theorem is a numerical formula of logarithmic Gromov-Witten invariants.

\begin{corollary} \label{numericalsplitting}
Follow the situation in Theorem \ref{main theorem}. Fix a subset $\bP$ of legs of the graph of $\btau$, which corresponds to a subset of punctured points. Let $\bP_i$ be the legs that lies in $\btau_i$ after splitting, for $i=1,\ldots,n$. There are evaluation maps $e: \ul{\scrM(X/B,\btau)}\rightarrow \ul{X}^{|\bP|}$ along punctured points in $\bP$ and $e_{\brho_i}: \ul{\scrM_{\brho_i}(X/B,\btau_i)}\rightarrow \ul{X}^{|\bP_i|}$ along punctured points $\bP_i$, for $\brho_i$ the $\btau_i$-marked decorated types. 

Let $\beta \in H^*(X^{|\bP|})$ be a cohomology class with a Künneth decomposition
\begin{equation*}
\beta = \sum _{\mu} \alpha_{\mu} \cdot \beta_{\mu,1} \boxtimes\ldots\boxtimes \beta_{\mu,r},
\end{equation*} 
where $\beta_{\mu,i} \in H^*(X^{|\bP_i|})$, for $i=1,\ldots,n$. Then, 
\begin{equation*}
\begin{split}
	& \delta_*\Big[\int_{[\scrM(X/B,\btau)]^{\virt}} e^*(\beta)\Big]  \\	 = \sum_{\mu} \sum_{[\brho] \in \Delta(\fV)} & \alpha_{\mu} \cdot \prod_{i=1}^r \frac{m_{[\brho]}}{|\Aut(\brho_i/\btau_i)|} \cdot  \int_{[\scrM(X/B,\brho_i)]^{\virt}} j_{\brho_i\btau_i}^*e_{\brho_i}^*(\beta_{\mu,i}).
\end{split}
\end{equation*}
\end{corollary}
\begin{proof}
The claim is a direct result of Theorem \ref{main theorem}, following the projection formula.
\end{proof}

\subsection{Idea of the proof and structure of the paper}
The foundation of the paper is based on the punctured Gromov-Witten invariants in \cite{ACGSPunc}. In Chapter \ref{background}, we provide a brief review of punctured Gromov-Witten theory and the gluing formalism. We briefly cover the basic theory of the moduli spaces of punctured logarithmic maps and the virtual theory over the moduli of the maps to the relative Artin fans in Section \ref{evaluationlogstructuresection}. We study the evaluation log structures in Section \ref{logevmapsection}. There are canonical \textit{evaluation idealized structures} on the modified moduli spaces such that they are idealized log smooth (Proposition \ref{tildeevisidealizedlogsmooth}). In Section \ref{gluingformalismsection}, we recall the gluing formalism studied in \cite[\S 5.2]{ACGSPunc}. It is shown in Proposition \ref{gluedmodulireducedsetup} that up to a reduction of the moduli spaces, it is sufficient to study the commutative diagram 
\begin{equation}  \label{globalgluingdigaram}
\begin{tikzcd}
\wt{\fM}^{\gl,\ev}_{\red} \arrow[r,"\delta_{\red}^{\ev}"] \arrow[d,"\ev"] & \prod_{i=1}^r \wt{\fM}_{\btau_i,\red}^{\ev} \arrow[d, "\prod \ev_{\tau_i}"] \\ X_{\btau} \arrow[r,"\Delta_X"] & \prod_{i=1}^r X_{\btau_i}
\end{tikzcd}
\end{equation}
following the fiber diagram \eqref{gluingfiberdiagram}.

Such diagram has nice properties. First, the moduli spaces $\prod_{i=1}^r \wt{\fM}_{\btau_i,\red}^{\ev}$ and the evaluation maps are both idealized log smooth (Prop \ref{tildeevisidealizedlogsmooth}, Cor \ref{evisidealizedlogsmooth}). Hence, locally they admit charts of toric morphisms. Second, under Assumption \ref{toricassumption}, the gluing strata $X_{\btau}$ and $X_{\btau_i}$ have global toric structures. The global toric structures provide a canonical patching of the splitting formulas from the local charts.

Since the local gluings are toric, we review the intersection theory in toric varieties in Chapter \ref{fssection} following \cite{FS97}. We give the necessary generalization of the Fulton-Sturmfels formula to toric stacks in Corollary \ref{keyquotientfs}.

The local form of the splitting formula is explored in detail in the first two sections of Chapter \ref{proofofthemaintheorem}. In Section \ref{toricstrata}, we study the structures of the gluing strata $X_{\btau}$ and $X_{\btau_i}$, based on the logarithmic fiber products of toric varieties studied in Appendix \ref{AppendixB}. Each of them is a disjoint union of log schemes isomorphic with each other, denoted $Z_{\btau}$ and $Z_{\btau_i}$ correspondingly. In Section \ref{localmodelsubsection}, we study the local chart of the gluing formalism \eqref{globalgluingdigaram}. \Etale locally, the moduli space $\prod_{i=1}^r \wt{\fM}_{\btau_i,\red}^{\ev}$ admits a smooth map to a quotient stack of a toric variety by an algebraic torus. The quotient stack $\cA^{\ev}$ is an evaluation enhancement of the Artin cone defined in \eqref{Aevdefinition} and has a canonical evaluation map to the gluing strata $\prod_{i=1}^r X_{\btau_i}$. By studying the gluing of $\cA^{\ev}$ using the generalized Fulton-Sturmfels formula, we obtain the splitting formula for $\cA^{\ev}$ in Lemma \ref{pushforwardforAglev} and the local splitting formula in Proposition \ref{localsplittingformula}. In Section \ref{gluingofthelocalmodels}, we finish the proof of the global splitting formula \ref{main theorem} (Proof.\ref{proofdetail}) by showing the splitting formula patches under a fixed generic displacement vector.

\subsection{Other approaches}
Relative Gromov-Witten invariants for smooth pairs $(X,D)$, studied in \cite{AMLYBR2001}\cite{Ionel2000TheSS}\cite{li2001}\cite{li2002}, are defined through the moduli spaces of stable maps to expansions of $X$ along $D$. The stable maps to the expansions are transverse, hence the degeneration formulas are obtained by gluing the underlying stable maps. Using the idea of expansion, the degeneration formulas for smooth pairs are studied using twisted stable maps in \cite{orbifoldtech} and logarithmic stable maps in \cite{kim2010}\cite{Chen2010TheDF}. These different approaches are proved to be identical with logarithmic Gromov-Witten invariants for smooth pairs in \cite{AIF_2014__64_4_1611_0}. In \cite{Kim2018TheDF}, Kim, Lho and
Ruddat provided a proof of the gluing formula for logarithmic Gromov-Witten invariants for smooth pairs without expansions using logarithmic technique. Because of the transverse nature of the underlying tropical geometry, all these approaches come with splitting formulas according to strata similar to our Corollary \ref{numericalsplitting}.

Combining the idea of expanded degenerations and tropical geometry, Ranganathan showed a general gluing formula of log Gromov-Witten invariants in the normal crossing settings in \cite{ranganathan2020logarithmic}. The numerical degeneration formula there requires the knowledge of a K\"{u}nneth decomposition of universal divisor expansions. We expect a similar splitting formula as we present can be obtained by proving an explicit K\"unneth formula for universal expansions of toric varieties. 

The gluing and splitting formalism using punctured Gromov-Witten invariants has a symplectic parallel by the theory of exploded manifolds due to Brett Parker in \cite{ParkerBrett2017TDTC} \cite{parker2017tropical}. The concept of generic deformation vectors in \cite{ParkerBrett2017TDTC} partially inspires our definition of the generic displacement vectors here. In a special case for rigid analytic Gromov-Witten invariants, a gluing formula has been proved by Yu \cite{yu2020enumeration}.

\color{black}
\subsection{Acknowledgement}
This work was supported by NSF grant DMS-1903437. I can not be more thankful to my advisor Bernd Siebert, who suggested this project and provided profound support, discussions and explanations. I thank Mark Gross for carefully reading and providing comments, corrections and simplifications of proofs to the paper. I would like to thank the referees for pointing out mistakes and many helpful suggestions. I thank Michel van Garrel for the opportunity to present this work in 3CinG Workshop 2020. I thank Lawrence Barrott, Rok Gregoric, Suraj Dash, Thomas Gannon and Yan Zhou for useful discussions.

\subsection{Conventions}
We follow the conventions in \cite{abramovich2020decomposition} and \cite{ACGSPunc}. All logarithmic schemes and stacks are fine and defined over an algebraically closed field $\kk$ over characteristic $0$. 

The affine log scheme with a global chart defined by a homomorphism $Q\rightarrow R$ from a monoid $Q$ to a ring $R$ is denoted $\Spec(Q\rightarrow R)$. For $Q$ a toric monoid, we define $Q^{\vee}: = \operatorname{Hom}(Q,\NN)$ and $Q^*: = \operatorname{Hom}(Q,\ZZ)$. We use $S_Q$ to denote affine toric variety $\spec(\kk[Q])$ with the canonical toric log structure and $T_Q$ to denote $\spec(\kk[Q^{\gp}])$. We define $\A_Q:=[S_Q/T_Q]$ to be the Artin cone of $Q$. Suppose $L\subseteq Q$ is an ideal of $Q$, then we use $S_{Q,L}$ to denote that subscheme of $S_Q$ determined by the ideal generated by $L$. We use $\cA_{Q,L}$ for the stack $[S_{Q,L}/T_Q]$.

For a toric variety $X$ and a cone $\sigma \in \Sigma(X)$ in the fan of $X$, we use $\cO_X(\sigma)$ to denote the algebraic torus $\spec \kk[\sigma^{\perp} \cap M]$ that is a subscheme of $X$, and we use $V_X(\sigma)$ to denote the closure of $\cO_X(\sigma)$ in $X$. For a Zariski log scheme $X$ and a cone $\sigma$ in the tropicalization of $X$, we use $V_X(\sigma)$ to denote the closed stratum whose dual cone of the stalk $\ol{\cM}_X$ at the generic point of $V_X(\sigma)$ is $\sigma$. For a logarithmic stack $X$ and a cone $\sigma$  in the tropicalization of $X$, we use $V_X(\sigma)$ to denote the strict closed integral substack with pullback $V_W(\sigma)$ on each Zariski smooth chart $W\rightarrow X$.
\textcolor{black}{
For a proper, representable morphism between logarithmic integral stacks $f: X\rightarrow Y$, we use $f_*[X]$ to denote the pushforward class $\ul{f}_*[\ul{X}]$, as studied in \cite[Def 3.6]{Vistoli1989} and \cite[Appendix B]{bae_schmitt_skowera_2022}.
}
We use $|\bS|$ to denote the cardinality of a finite set $\bS$.

\section{Punctured invariants and the gluing formalism} \label{background}
\changelocaltocdepth{2}
In this section, we give a brief introduction to the punctured Gromov-Witten invariants and the gluing formalism studied in \cite{ACGSPunc}. We show the gluing formalism admits a local model of fiber product of toric varieties in the category of fine, saturated log schemes. 
\subsection{Punctured Gromov-Witten invariants}

\label{evaluationlogstructuresection}

Let $X$ be a projective log smooth scheme over a log scheme $B$. A \textit{punctured log curve} over a log scheme $W$ is given by
\begin{equation*}
(C^{\circ}\xrightarrow{p} C \xrightarrow{\pi} W, \bp=(p_1,\ldots,p_n)),
\end{equation*}
where 
\begin{enumerate}
\item $C\rightarrow W$ is a logarithmic curve with a set of disjoint sections $\{p_1,\ldots,p_n\}$.
\item $C^{\circ}$ is a logarithmic curve with the underlying curve $\ul{C}$ and log structure 
\begin{equation*}
\cM_{C^{\circ}}\subset \cM_C\oplus_{\cO_C^{\times}} \cP^{\gp}
\end{equation*}
for $\cP\subset \cM_C$ the divisorial log structure along sections $\bp$, such that for any geometric point $\ol{x}\in \ul{C}$ and $s_{\ol{x}}\notin \cM_{\ol{x}}\oplus_{\cO_C^{\times}
} \cP_{\ol{x}}$, we have $\alpha_{C^{\circ}}(s_{\ol{x}}) = 0$. 
\end{enumerate}
We note that $\cM_{C^{\circ}}$ is not necessarily saturated. Figure $1$ in \cite{ACGSPunc} provides a nice example. 
A \textit{punctured log map to} $X\rightarrow B$ over $W\rightarrow B$ is a punctured log curve $(C^{\circ}\rightarrow C \rightarrow W, \bp)$ and a morphism $f:C^{\circ} \rightarrow X$ over $B$. It is \textit{stable} if $\cM_{C^{\circ}}$ is generated by $\cM_C$ and $f^{\flat}(f^*(\cM_X))$ and the underlying map $\ul{f}$ is stable in the usual sense.

The \textit{contact orders} of a punctured map over a log point $W=\spec(Q\rightarrow \kk)$ at point $p\in\bp$ is the composition 
\begin{equation*}
u_p: \ol{\cM}_{X,\ul{f}(p)} \xrightarrow{f^{\flat}} \ol{\cM}_{C,p} \rightarrow Q\oplus \ZZ \xrightarrow{\pr_2} \ZZ.
\end{equation*} The contact order is \textit{negative} if the image of $u_p$ is not contained in $\NN$, which naturally occurs over the points $p$ with $\cM_{C,p}$ a strict submonoid of $\cM_{C^{\circ},p}$. 

{\color{black}
Via the functoriality of the tropicalization functor, a stable punctured log maps gives rise to a family of tropical punctured maps (\cite[\S 2.2.1]{ACGSPunc}), where we extract the combonatorical data of \textit{global types}. As in the theory of logarithmic Gromov-Witten, the moduli spaces of the stable punctured log maps to $X$ are stratified by global types.}

\begin{definition}\cite[Def. 3.4]{ACGSPunc} \label{globaltypedefinition}
A \textit{global type $\tau$} of a family of tropical punctured maps is a tuple $(G,\bg, \ol{\bu}, \pmb{\sigma})$ consisting of 
\begin{enumerate}
    \item A connected graph $G$ with a set of vertices $V(G)$, a set of edges $E(G)$ and a set of legs $L(G)$.
    \item A genus map $\bg: V(G)\rightarrow \NN$.
    \item An image cone map $\bsigma: V(G)\cup E(G) \cup L(G)\rightarrow \Sigma(X)$.
    \item A global contact order map $\ol{\bu}$ 
    \begin{equation*}
        \ol{\bu}: E(G)\cup L(G) \rightarrow \bigsqcup_{\sigma\in \Sigma(X)} \fC_{\sigma}(X)
    \end{equation*}
    such that $\ol{\bu}(x) \in \fC_{\bsigma(x)}(X)$, with $\bsigma(x)$ the image cone of any edge or leg $x$. Here, for any cone $\sigma \in \Sigma(X)$, we define 
    \begin{equation*}
        \fC_{\sigma}(X) : = \operatorname{colim}_{y\in V_X({\sigma})}^{\operatorname{\pmb{Sets}}}{N_{\sigma_y}}.
    \end{equation*}
\end{enumerate}
for a point $y\in X$. By the cone $\sigma_y$ we mean the dual cone $\ol{\cM}^{\vee}_{X,y}$.

A \textit{global decorated type} $\btau$ is a tuple $(\tau, \bA)$ with $\tau$ a global type and $\bA$ a function from $V(G)$ to a monoid of curve classes of $X$. We say a global type $\tau$ or a global decorated type $\btau = (\tau, \bA)$ is realizable if there exists a tropical map to $\Sigma(X)$ with associated global type $\tau$.
\end{definition}

A \textit{marking by $\btau$} of a punctured map $(C^{\circ}/W,\bp,f)$ is defined in \cite[Def.3.7]{ACGSPunc}. Roughly speaking, a map is marked by $\btau$ if the genus decorated dual graph of the curve $C$ admits a contraction to $(G_{\btau},\bg_{\btau})$, the image of each nodes and punctured points lies in the associated logarithmic strata of the cone $\bsigma$, both the contact orders of non-contracted edges and legs and the curve classes after contraction are determined by $\btau$. The following theorem in \cite{ACGSPunc} lays the foundation of the punctured Gromov-Witten theory. 

\begin{theorem}\cite[Thm A]{ACGSPunc}
Let $\btau$ be a global decorated type. Then the moduli space $\scrM(X/B,\btau)$ of $\btau$-marked basic stable punctured maps to $X\rightarrow B$ is a Deligne-Mumford logarithmic algebraic stack and is proper over $B$.  
\end{theorem}

The insights of Olsson's category of logarithmic schemes \cite{OlssonMartinC2003Lgaa} lead to the concept of Artin fans. As defined in \cite[\S 2.2]{abramovich2020decomposition}, for a log Deligne-Mumford stack $X$, the Artin fan of $X$ is the algebraic stack constructed by gluing toric quotient stacks, called Artin cones, of stalks of $\ol{\cM}_X$. Let $\ol{x}$ be a geometric point on $X$ and let $P_{\ol{x}}$ be $\ol{\cM}_{X,\ol{x}}$. We define an Artin cone $\cA_{\ol{x}} = [\spec \kk[P_{\ol{x}}]/\spec \kk[P_{\ol{x}}^{\gp}]]$. The generization of points results in open embeddings of Artin cones. The Artin fan $\cA_X$ is the colimit of Artin cones along all points. Artin fans play an important role in the virtual theory and connect the tropical picture with the log picture.

Let $\X = \A_X \times_{\A_B} B$ be the relative Artin fan. The moduli space $\fM(\X/B,\btau)$ of $\btau$-marked basic stable punctured maps to $\X\rightarrow B$ is again an algebraic stack. For $\btau$ realizable, the moduli space $\fM(\X/B,\btau)$ is pure dimensional(\cite[Prop.3.28]{ACGSPunc}).

There is a natural evaluation map
\begin{equation*}
\fM(\X/B,\btau) \rightarrow \ul{\X}\times_{\ul{B}}\ldots\times_{\ul{B}} \ul{\X},
\end{equation*}
taken over all the edges and legs of type $\btau$. We define
\begin{equation} \label{fmevdefinition}
    \fM^{\ev}(\X/B,\btau) = \fM(\X/B,\btau) \times_{(\ul{\X}\times_{\ul{B}}\ldots\times_{\ul{B}} \ul{\X})} (\ul{X}\times_{\ul{B}}\ldots\times_{\ul{B}} \ul{X}).
\end{equation} 

Let $\bS$ be a subset of edges of the graph $G$ of $\btau$. By splitting $G$ along the edges in $\bS$, we obtain a collection of types $\btau_i$, $i=1,\ldots,r$. As shown by the following theorem, the virtual theory of the splitting morphism of the moduli spaces of punctured maps to $X\rightarrow B$ is compatible with the splitting morphism of the moduli spaces of punctured maps to the relative Artin fans $\X\rightarrow B$.

\begin{theorem} \cite[Thm C, Prop 5.15, Thm 5.17]{ACGSPunc} \label{setuptheorem}

There is a Cartesian diagram
 \begin{equation} \label{setupdiagram}
 \begin{tikzcd}
     \scrM(X/B,\btau) \arrow[r, "\delta"] \arrow[d,"\hat{\varepsilon}"] & \prod_{i=1}^r \scrM(X/B,\btau_i) \arrow[d,"\varepsilon"] \\
     \fM^{\ev}(\X/B,\btau) \arrow[r,"\delta'"] & \prod_{i=1}^r \fM^{\ev}(\X/B,\btau_i),
\end{tikzcd}
\end{equation}
with horizontal splitting maps finite and representable, and vertical maps strict morphisms.  There are obstruction theories 
\begin{equation*}
\begin{split}
    \GG & \rightarrow \LL_{\scrM(X/B,\btau)/\fM^{\ev}(\X/B,\btau)} \\
    \GG_{\spl{}} & \rightarrow \LL_{\prod_{i=1}^r \scrM(X/B,\btau_i)/\prod_{i=1}^r \fM^{\ev}(\X/B,\btau_i)},
\end{split}
\end{equation*}
such that the obstruction theory of the left vertical map is the pullback of the obstruction theory of the right vertical map. For $\alpha \in A_*(\fM^{\ev}(\X/B,\btau))$, there is
\begin{equation*}
    \delta_*\hat{\varepsilon}^{!}(\alpha) = \varepsilon^!\delta_*'(\alpha),
\end{equation*}
where $\hat{\varepsilon}^!$ and $\varepsilon^!$ are the Manolache's virtual pullback defined using these two obstruction theories.
\end{theorem}

\subsection{Logarithmic evaluation maps} \label{logevmapsection}
Different from Jun Li's situation using expanded degenerations, the gluing of a logarithmic stable map from the restrictions to closed subcurves requires more than gluing on the schematic level. In order to obtain a gluing formalism, we first need to fix the problem of the non-existence a logarithmic evaluation map from the moduli space $\fM(\X/B,\btau)$ to $\X$. It requires us to do a modification of the log structure on the moduli space. 

For ease of notation, we use $\fM_{\btau}: = \fM(\X/B,\btau)$ and $\MF_{\btau}^{\ev}: = \MF^{\ev}(\X/B,\btau)$ for the rest of the paper. Let $G$ be the graph associated to $\btau$. For each element $p \in E(G) \cup L(G)$, let $\ul{s}_p: \ul{\fM}_{\btau} \rightarrow \ul{\fC^{\circ}}$ be the universal section of the punctured or nodal point associated to $p$. Define $\widetilde{\MF}_p$ to be the logarithmic algebraic stack with the underlying stack $\underline{\MF}_{\btau}$ and the log structure $\ul{s}_p^*\M_{\mathfrak{C}^{\circ}}$. With this log structure, there is a canonical evaluation map $\wt{\fM}_p\rightarrow \X$ on the section of $p$. Note that the log structure on $\wt{\fM}_p$ is fine, but may not be saturated.

For a subset $\bS \subseteq E(G)\cup L(G)$, let $\widetilde{\MF}_{\bS, \btau}$ be the saturation of the fine fiber product
\begin{equation} \label{Mtildedefinition}
    \widetilde{\MF}_{p_1} 
    \times^{\operatorname{fine}}_{\MF_{\btau}} \ldots\times^{\operatorname{fine}}_{\MF_{\btau}}\widetilde{\MF}_{p_{|\bS|}}, \quad p_i\in \bS
\end{equation}
in the category of fine log stacks. For the rest of the section, we fix a subset $\bS$ and use $\wt{\fM}_{\btau}$ for $\wt{\fM}_{\bS,\btau}$. Define $\wt{\fM}^{\ev}_{\btau} = \widetilde{\MF}_{\btau} \times_{\fM_{\btau}} \fM_{\btau}^{\ev}$.

\begin{proposition} \label{isoafterenhanced}
\cite[Prop.5.5]{ACGSPunc}
The canonical map $\wt{\fM}_{\btau}^{\ev} \rightarrow \fM^{\ev}_{\btau}$ is an isomorphism on the underlying stacks provided $\bS \subseteq E(G)$, and generally induces an isomorphism on the reductions.
\end{proposition}

There is a \textit{canonical idealized structure} on $\fM_{\btau}$, such that $\fM_{\btau}$ is idealized log smooth (\cite[Thm 3.24]{ACGSPunc}). The idealized structure in \cite[Def 3.22]{ACGSPunc} comes from the fixed combinatorial conditions including dual graph $G$, image strata fixed by $\bsigma$, the contraction to the global type $\btau$ and the puncturing ideal. We will construct an \textit{evaluation idealized structure} on $\wt{\fM}_{\btau}$ following Construction \ref{constructionidealizedconstruction}, with which $\wt{\fM}_{\btau}$ is also idealized log smooth. Similar to the log structures on $\fM_{\btau}$, both log structures and evaluation idealized structure on $\wt{\fM}_{\btau}$ are determined by the global type $\btau$. 

For a basic punctured log map of type $\btau$ over a point $w$, the dual cone of the stalk $(\ol{\cM}_{\fM_{\btau},\ol{w}})^{\vee}_{\RR}$, is called \textit{associated basic cone} of $\btau$. The associated basic cone parametrizes the tropical maps of type $\btau$, which we describe concretely in the following Definition \ref{evaluationcone}. Similarly, the dual cone of the stalk $(\ol{\cM}_{\wt{\fM}_{\btau},\ol{w}})^{\vee}_{\RR}$ also admits a simple description by \textit{associated evaluation cone} of $\btau$, which parametrizes the tropical maps of type $\btau$ with an additional marking on each edge or leg in $\bS$.

\begin{definition} \label{evaluationcone}
Let $\btau$ be a realizable global decorated type. Define the \textit{associated basic cone $\ol{\btau}$ of $\btau$} the set of elements \begin{equation*}
\begin{split}
    ((V_v)_{v\in V(G)}, (l_E)_{E\in E(G)}) \in  \prod_{v\in V(G)} \bsigma(v) \times \prod_{E\in E(G)} \R_{\geq 0},
\end{split}
\end{equation*}
such that $V_{v_E}-V_{v_E'} = l_E \cdot \ol{\bu}(E)$. Here $v_E$ and $v_E'$ are the vertices of the edge $E$, with order specified by $\ol{\bu}(E)$. As $V_{v_E}$ and $V_{v_E'}$ both lie in $\bsigma(E)$, the difference $V_{v_E}-V_{v'_E}$ is well-defined.

Define the \textit{asscociated evaluation cone} $\wt{\btau}_{\bS}$ of $\btau$ with respect to a set $\bS\subseteq E(G)\cup L(G)$ to be the set of elements
\begin{equation*}
\begin{split}
    ((V_v)_{v\in V(G)}, (l_E)_{E\in E(G)},(t_p)_{p\in \bS}) \in  \prod_{v\in V(G)} \bsigma(v) \times \prod_{E\in E(G)} \R_{\geq 0} \times \prod_{p\in \bS} \R_{\geq 0},
\end{split}
\end{equation*}
such that $V_{v_E}-V_{v_E'} = l_E \cdot \ol{\bu}(E)$, $V_{v_p} + t_p \cdot \ol{\bu}(p) \in \bsigma(p)$ and $t_e\leq l_e$ for $e$ in $E(G)\cap \bS$. Here, if $p\in L(G)$, we define the vertex $v_p$ to be the vertex of leg $p$; if $p$ is an edge $E\in E(G)$, we define the vertex $v_p$ to be the vertex $v_{E}'$ with $V_{v_E}-V_{v_E'} = l_E \cdot \ol{\bu}(E)$, specified by the orientation of the contact order. There is a \textit{tropical evaluation map} 
\begin{equation*}
\begin{split}
\evt_{\btau} : \wt{\btau}_{\bS}  \rightarrow \prod_{p\in \bS}\bsigma(p) & ,\\
((V_v)_{v\in V(G)}, (l_E)_{E\in E(G)}, (t_p)_{p\in \bS}) & \mapsto (V_{v_p}+t_p\cdot \ol{\bu}(p))_{p\in \bS}.
\end{split}
\end{equation*}
\end{definition}

Under the case of $B = \spec (Q_B\rightarrow \kk)$, the tropical evaluation map $\evt_{\btau}$ factors through the fiber product of cones $\bsigma(p)$ over $Q_B^{\vee}$. The map
\begin{equation} \label{evdiscussion}
    \ev_{\btau}: \wt{\btau}_{\bS} \rightarrow \bsigma(p_1)\times_{Q_B^{\vee}}\ldots\times_{Q_B^{\vee}}\bsigma(p_{|\bS|})
\end{equation}
that $\evt_{\btau}$ factors through is later used in Lemma \ref{pushforwardforAglev}. 

\begin{lemma} \label{evaluationconetype} 
Let $\btau$ be the tropical type of the punctured map over a geometric point $\ol{w}$ on $\wt{\fM}_{\btau}$. Then, there is an isomorphism between the dual cone $(\ol{\cM}_{\wt{\fM}_{\btau},\ol{w}}^{\vee})_{\RR}$ and $\wt{\btau}_{\bS}$.
\end{lemma}

\begin{proof}
By the definition of $\wt{\fM}_{\btau}$ in \eqref{Mtildedefinition}, there is a projection from $\wt{\fM}_{\btau} \rightarrow \wt{\fM}_p$, for every $p\in \bS$. Let $\ol{w}_p$ be the geometric point in $\wt{\fM}_p$ under the projection. Let $\wt{Q}_p$ be the monoid $\ol{\cM}_{\wt{\fM}_p,\ol{w}_p}$. From the tropical interpretation of the basic log structure in \cite[§2.2]{ACGSPunc}, for $p\in \bS$ and the associated punctured or nodal point $\ul{s}_p:\ul{\fM}_{\btau} \rightarrow \ul{\fC}^{\circ}$, there are isomorphisms between $\wt{Q}^{\vee}_{p,\RR} = \ul{s}_p^*(\ol{\cM}_{{\fC}^{\circ},\ul{s}_p(\ol{w})})$ and the cone
\begin{equation*}
    ((V_v)_{v\in V(G)}, (l_E)_{E\in E(G)},t_p) \in  \prod_{v\in V(G)} \bsigma(v) \times \prod_{E\in E(G)} \R_{\geq 0} \times \R_{\geq 0}
\end{equation*}
with $V_{v_E}-V_{v_E'} = l_E \cdot \ol{\bu}(E)$, $V_{v_p} + t_p \cdot \ol{\bu}(p) \in \bsigma(p)$ and $t_p\leq l_p$ if $p\in E(G)$.

As $\wt{Q}$ is the saturation of $\wt{Q}_{p_1}\oplus_{Q}\ldots\oplus_Q \wt{Q}_{p_{|\bS|}}$, the dual cone $\wt{Q}_{\RR}^{\vee}$ is the fiber product of cones $\wt{Q}^{\vee}_{p_1,\RR}\times_{Q^{\vee}_{\RR}}\ldots\times_{Q^{\vee}_{\RR}} \wt{Q}^{\vee}_{p_{|\bS|,\RR}}$. Thus, there is an isomorphism of cones $\wt{Q}_{\RR}^{\vee}\rightarrow \wt{\btau}_{\bS}$.
\end{proof}

Now we construct the idealized structure on $\wt{\fM}_{\btau}$, with which $\wt{\fM}_{\btau}$ is idealized log smooth over $B$. It follows from the following general construction of an idealized log structure on a logarithmic stack $M$, {\color{black} assuming there is a strict closed embedding of $(M,\cM_M) \rightarrow (N,\cM_N)$ determined by a sheaf of log ideals of $\cM_N$ and there is an idealized log structure $\cK_N$ on $N$. The construction is the same as the log scheme case in \cite[Prop III.1.3.4]{LogAlgebraicGeometry}.
}
\begin{construction} \label{constructionidealizedconstruction}
Assume there is a strict closed embedding of $(M,\cM_M$) $\rightarrow (N,\cM_N)$, such that $M$ is the closed substack of $N$ determined by the ideal generated by $\alpha_N(\cK')$, with $\cK'$ a log ideal sheaf of $\cM_N$ and $\alpha_N$ the structure morphism. We construct an idealized structure $\cK_M$ to be the ideal sheaf of $\cM_{M}$ generated by the pullback of $\cK'$ and $\cK_N$. 
\end{construction}

It is easy to check that $\alpha_M(\cK_M) = 0$, thus $\cK_M$ is a well-defined idealized structure. The morphism $(M,\cM_M,\cK_M)\rightarrow (N,\cM_N,\cK_N)$ is idealized log smooth by \cite[Variant IV.3.1.21]{LogAlgebraicGeometry}. Note that by definition, a logarithmic stratum $M$ of $N$ is determined by a logarithmic ideal sheaf, hence satisfies the condition for the construction.

Let $\cK_{\fM_{\btau}}$ be the \textit{canonical idealized structure} on ${\fM}_{\btau}$ defined in \cite[Def.3.22]{ACGSPunc}. For $p\in \bS$, the map $e_p: \wt{\fM}_p\rightarrow \fM_{\btau}$ is the composition of the strict section map $s_p: \wt{\fM}_p \rightarrow \fC^{\circ}$ and the universal curve $\fC^{\circ} \rightarrow \fM_{\btau}$. The section map $\wt{\fM}_p \rightarrow \fC^{\circ}$ is a closed immersion of the logarithmic stratum of $\fC^{\circ}$ associated to the puncturing $p$. We define an idealized log structure $\cK_{\fC^{\circ}}$ on $\fC^{\circ}$ by the pullback of $\cK_{\fM_{\btau}}$ on $\fM_{\btau}$. Define $\cK_{\wt{\fM}_p}$ to be the canonical idealized structure on $\wt{\fM}_p$ associated to $s_p$ constructed in Construction \ref{constructionidealizedconstruction} and $\K_{\wt{\MF}_{\bS}}$ to be the sheaf of ideals generated by the pullbacks of ideals $\cK_{\wt{\fM}_p}$ under the projection maps $\wt{\fM}_{\btau} \rightarrow \wt{\MF}_p$.

\begin{proposition} \label{tildeevisidealizedlogsmooth}
The logarithmic algebraic stack $\widetilde{\MF}_{\btau}$ with log ideal $\K_{\widetilde{\MF}_{\btau}}$ is idealized log smooth over $B$.
\end{proposition}

\begin{proof}
With the idealized structure $\cK_{\fC^{\circ}}$ on $\fC^{\circ}$, the universal curve $ \fC^{\circ}\rightarrow \fM_{\btau}$ is ideally strict, that is, the idealized structure on $\fC^{\circ}$ is generated by the pullback of idealized structure on $\fM_{\btau}$. Since $\fC^{\circ}\rightarrow \fM_{\btau}$ is log smooth, it is idealized log smooth by \cite[Variant IV.3.1.22]{LogAlgebraicGeometry}.  By \cite[Variant IV.3.1.21]{LogAlgebraicGeometry}, the closed embedding $\wt{\fM}_p\rightarrow \fC^{\circ}$ is idealized log smooth. Hence, we obtain that $e_p: \wt{\MF}_p \xrightarrow{s_p} \fC^{\circ} \xrightarrow{} \MF_{\btau}$ is idealized log smooth.

Let $\wt{\fM}^{\operatorname{fine}}_{\btau}$ be the fiber product of fine logarithmic stacks
\begin{equation*}
    \widetilde{\MF}_{p_1} 
    \times^{\operatorname{fine}}_{\MF_{\btau}} \ldots\times^{\operatorname{fine}}_{\MF_{\btau}}\widetilde{\MF}_{p_{|\bS|}},
\end{equation*}
with $p_i$ going over elements in $\bS$. As the idealized log smoothness is stable under fine fiber products, with ideal sheaf $\cK_{\wt{\fM}^{\operatorname{fine}}_{\btau}}$ on $\wt{\fM}^{\operatorname{fine}}_{\btau}$ generated by the pullback of ideals $\cK_{\wt{\fM}_p}$, the projection map  $\wt{\fM}^{\operatorname{fine}}_{\btau}\rightarrow \fM_{\btau}$ is idealized log smooth. By the idealized log smoothness of $\fM_{\btau}$ over $B$, we obtain that $\wt{\fM}^{\operatorname{fine}}_{\btau}$ is idealized log smooth over $B$.

Let $g: \wt{\fM}_{\btau}\rightarrow \wt{\fM}^{\operatorname{fine}}_{\btau}$ be the saturation morphism. By \cite[§III.3.1.11]{LogAlgebraicGeometry}, the saturation morphism $g$ is log \etale. As the projection maps $\wt{\fM}_{\btau} \rightarrow \wt{\fM}_p$ factor through $g$, the ideal sheaf $\cK_{\wt{\fM}_{\btau}}$ is generated by $g^*(\cK_{\wt{\fM}^{\operatorname{fine}}_{\btau}})$. The morphism $g$ is ideally strict, hence is idealized log smooth. The logarithmic algebraic stack $\wt{\fM}_{\btau}$ is idealized log smooth over $B$. 

\end{proof}

\begin{corollary} \label{evisidealizedlogsmooth}
Let $\wt{\fM}_{\btau,\red}^{\ev}$ be the reduced induced logarithmic stack of $\wt{\fM}^{\ev}_{\btau}$. Let $\cK_{\wt{\fM}_{\btau,\red}^{\ev}}$ be the idealized structure on $\wt{\fM}^{\ev}_{\btau,\red}$ associated to the strict closed embedding to $\wt{\fM}^{\ev}_{\btau,\red}\rightarrow \wt{\fM}_{\btau}^{\ev}$ constructed in Construction \ref{constructionidealizedconstruction}. Then the corresponding idealized log stack $\wt{\fM}_{\btau,\red}^{\ev}$ is idealized log smooth over $B$.
\end{corollary}

\begin{proof}
The statement follows from \cite[Variant IV.3.1.21]{LogAlgebraicGeometry}.
\end{proof}

Following the idealized smoothness of $\wt{\fM}_{\btau}$, we obtain that the stratification of $\wt{\fM}_{\btau}$ is encoded in the global types with \textit{contraction morphisms} to $\btau$, similar to \cite[Rmk 3.29]{ACGSPunc}. 

\begin{definition} \label{contraction morphism}
A \textit{contraction morphism} of global decorated types $\bomega \rightarrow \btau$ is a map of the graphs $G_{\omega}\rightarrow G_{\tau}$ contracting a subset of edges, such that, the following properties are satisfied: \begin{enumerate}
    \item the global contact order of $\btau$ associated to an edge or a leg $p$ is the same as the global contact order of the edge or leg in $\bomega$ surjective onto $p$,
    \item the genus and the curve class of a vertex in $\btau$ is the sum of those of the vertices in $\bomega$ mapped to $v$ and 
    \item the cone of a vertex, edge or leg in $\btau$ is a subcone of any vertices, edges or legs contained in the preimage.
\end{enumerate}
\end{definition}

Suppose $\bomega\rightarrow \btau$ is a contraction of global decorated type. The preimage of elements in $\bS$ form a subset of the edges and legs of $\bomega$, which we again denote $\bS$. Then, by the definition of the associated evaluation cones in \ref{evaluationcone}, there is a face inclusion $\wt{\tau}_{\bS} \rightarrow \wt{\omega}_{\bS}$ whose image is the locus corresponding to points with $l_E = 0$ for the contracted edges in the graph of $G$. The evaluation map $\evt_{\btau}$ in Definition \ref{evaluationcone} is the restriction of $\evt_{\bomega}$ on $\wt{\btau}_{\bS}$.

\begin{remark} \label{Idealstructureremark}
Let us give the idealized structure on $\wt{\fM}^{\ev}_{\btau,\red}$ a local description. 

Let $\bomega$ be a global decorated type that admits a contraction morphism to $\btau$. We first take a look at the idealized structure of $\fM_{\bomega}$. Let $Q_{\bomega} = \operatorname{Hom}(\bomega_{\ZZ},\NN)$ be the associated basic monoid of $\bomega$ as defined in \cite[Def 2.38]{ACGSPunc} and $Q_{\btau} = \operatorname{Hom}({\btau}_{\ZZ},\NN)$ be the associated basic monoid of $\btau$. Let $L_{\bomega}$ be the stalk of the ideal sheaf $\ol{\cK}_{\fM_{\btau}}$ at a geometric point of type $\bomega$. Since $\btau$ is realizable, by \cite[Prop.3.23]{ACGSPunc}), the ideal $L_{\bomega}$ is generated by the inverse image of $Q_{\btau}\backslash \{0\}$ under the generization map $Q_{\bomega}\rightarrow Q_{\btau}$. 

Next, we take a look at the local structure of punctured points. Let $Q_{\bomega,p}\subseteq Q_{\bomega}\oplus \ZZ$ be the stalk of $\ol{\cM}_{\fC^{\circ}}$ at the punctured or nodal point associated to $p\in \bS$ of a punctured map with type $\bomega$. Let $L_{\bomega,p}$ be the ideal generated by the preimage of $L_{\bomega}$ under $Q_{\bomega,p}\rightarrow Q_{\bomega}$ and the ideal $Q_{\bomega,p} \cap (Q_{\bomega}\oplus \ZZ_{>0})$. It follows that $L_{\bomega,p}$ is generated by the preimage of $L_{\btau,p}$ under the generization map $Q_{\bomega,p} \rightarrow Q_{\btau,p}$, thus is generated by the preimage of $Q_{\btau,p} \backslash \{0\}$. 

Now, we are ready to study the idealized structure of $\wt{\fM}^{\ev}_{\btau}$. 
Let $\wt{Q}_{\bomega} = \operatorname{Hom}(\wt{\bomega}_{\ZZ},\NN)$ and $\wt{L}_{\bomega}$ be the stalk of the ideal sheaf $\ol{\cK}_{\wt{\fM}^{\ev}_{\btau}}$ at the geometric point $\ol{x} \rightarrow \wt{\fM}_{\btau,\red}^{\ev} \rightarrow \wt{\fM}^{\ev}_{\btau}$. As the monoid $\wt{Q}_{\bomega}$ is the saturation of the fibered sum
\begin{equation*}
Q_{\bomega,p_1}\oplus_{Q_{\bomega}}\ldots\oplus_{Q_{\bomega}} Q_{\bomega,p_{|\bS|}}
\end{equation*}
in the category of fine monoids, the ideal $\wt{L}_{\bomega}$ is generated by the image of $L_{\bomega, p_i}$ together with the elements in $\wt{Q}_{\bomega}$ which are mapped to the nilpotent elements under the structure morphism. For type $\btau$, the ideal $\wt{L}_{\btau}$ admits a similar description. The ideal $L_{\btau,p} = Q_{\btau,p}\backslash \{0\}$, hence $\wt{L}_{\btau}$ is the prime ideal $\wt{Q}_{\btau} \backslash \{0\}$. As $L_{\bomega,p}$ is generated by the preimage of $L_{\btau,p}$, we obtain that $\wt{L}_{\bomega}$ is the preimage of $\wt{L}_{\btau}$. The toric variety
\begin{equation*}
\spec \kk[\wt{Q}_{\bomega}]/(\wt{L}_{\bomega}) = V_{\spec \kk[\wt{Q}_{\bomega}]}(\wt{\btau})
\end{equation*}
is the toric strata associated to the subcone $\wt{\btau}$ in $\wt{\bomega}$.
\end{remark}

\begin{corollary}
\label{stratatificationofmmm}
Let $\Sigma({\wt{\fM}_{\btau}})$ be the tropicalization of the Artin stack $\wt{\fM}_{\btau}$ as mentioned in \cite[§ 2.1.4]{abramovich2020decomposition} and constructed in \cite{AbramovichDan2015Ttot} and \cite{UlirschMartin2017Ftol}. Then, the image of the finite morphism $\wt{\fM}_{\bomega} \rightarrow \wt{\fM}_{\btau}$ is the substack $V_{\wt{\fM}_{\btau}}(\wt{\bomega}_{\bS})$ associated to the cone $\wt{\bomega}_{\bS} \in \Sigma(\wt{\fM}_{\btau})$.
\end{corollary}

\begin{proof}
It follows from the idealized smoothness of $\wt{\fM}_{\bomega}$ and $\wt{\fM_{\btau}}$ and the local description of the associated basic monoids and idealized structure in Remark \ref{Idealstructureremark}.
\end{proof}

\subsection{The gluing formalism} \label{gluingformalismsection}

Fix a decorated global tropical type $\btau = (\tau, \bA)$ with $\tau$ realizable and $\bS \subseteq E(G)$ a subset of edges of the graph $G$ of $\btau$. By splitting along edges in $\bS$, we obtain sub-types $\btau_1, \btau_2,\ldots,\btau_r$. For $i=1,2,\ldots,r$, let $\bS_i$ be the subset of legs of the graph in $\btau_i$, obtained from the splitting edges.

In the rest of the section, we use $\wt{\btau}$ and $\wt{\btau}_i$ to denote the evaluation cones $\wt{\btau}_{\bS}$ and $\wt{\btau}_{i,\bS_i}$. For a global decorated type $\bomega$ that admits a contraction to $\btau$, the set $\bS$ is a subset of edges of $\bomega$, we use $\wt{\bomega}$ to denote the evaluation cone $\wt{\bomega}_{\bS}$. Similarly, we use $\wt{\bomega}_i$ to denote the evaluation cone $\wt{\bomega}_{i,\bS_i}$ for $\bomega_i$ that admits a contraction to $\btau_i$.

In the previous section, we constructed the logarithmic evaluation map $\ev_p: \wt{\fM}_{\btau}^{\ev}\rightarrow X$ for each $p\in \bS$. The global type restricts the reduction of the image strata of $\ev_p$ to $V_X(\bsigma(p))$. We use $V_p$ to denote $V_X(\bsigma(p))$. Define 
\begin{equation} \label{Xtaudefinition}
X_{\btau}: = V_{p_1}\times_B^{\operatorname{fs}}\ldots\times^{\operatorname{fs}}_B V_{p_{|\bS|}},\quad p_j\in \bS
\end{equation} 
As $\wt{\fM}_{\btau}^{\ev}$ is reduced by \cite[Prop.3.28]{ACGSPunc}, we obtain an evaluation map $\ev_{\btau}$ from $\wt{\fM}_{\btau}^{\ev}$ to $X_{\btau}$.  Similarly, let 
\begin{equation*}
X_{\btau_i}:= V_{p_1}\times^{\operatorname{fs}}_B\ldots\times^{\operatorname{fs}}_B V_{p_{|\bS_i|}}, \quad p_j\in \bS_i.
\end{equation*} 
and $\ev_{\btau_i}$ be the corresponding evaluation map $\wt{\fM}_{\btau_i}^{\ev}\rightarrow X_{\btau_i}$.

Define $\wt{\fM}^{\gl,\ev}$ to be the following fiber product in the category of fine, saturated logarithmic stacks
\begin{equation} \label{gluingfiberdiagram}
\begin{tikzcd}
\wt{\fM}^{\gl,\ev} \arrow[r,"\delta^{\ev}"] \arrow[d,"\ev"] & \prod_{i=1}^r \wt{\fM}_{\btau_i}^{\ev} \arrow[d, "\prod \ev_{\tau_i}"] \\ X_{\btau} \arrow[r,"\Delta_X"] & \prod_{i=1}^r X_{\btau_i}.
\end{tikzcd}
\end{equation} The gluing formalism of \cite[Cor.5.13]{ACGSPunc} relates the fiber product $\wt{\fM}^{\gl,\ev}$ with $\wt{\fM}_{\btau}^{\ev}$. By the reducedness of $\wt{\fM}_{\btau}^{\ev}$ in \cite[Prop.3.28]{ACGSPunc}, we obtain the following Lemma.

\begin{lemma}\label{gluingreducedlemma}
Let $\wt{\fM}^{\gl,\ev}_{\red}$ be the reduction of logarithmic algebraic stack $\wt{\fM}^{\gl,\ev}$. Then, the morphism from $\wt{\fM}_{\btau}^{\ev}$ to $\wt{\fM}^{\gl,\ev}$ induced by the fiber diagram factors through the map $\wt{\fM}^{\gl,\ev}_{\red}\rightarrow \wt{\fM}^{\gl,\ev}$. Furthermore, it induces an isomorphism between $\wt{\fM}^{\gl,\ev}_{\red}$ and $\wt{\fM}_{\btau}^{\ev}$. 
\end{lemma}

Before we prove Lemma \ref{gluingreducedlemma}, let us first use it to show the main result of this section Proposition \ref{gluedmodulireducedsetup}, which implies that in order to study the pushforward of the virtual fundamental class under the splitting morphism \eqref{splittingmorphism}, it is enough to study the map $\delta^{\ev}$ in diagram \eqref{gluingfiberdiagram}. 

\begin{proposition} \label{gluedmodulireducedsetup}
Let $\gamma_{\btau}: \wt{\fM}^{\ev}_{\btau}\rightarrow \fM^{\ev}_{\btau}$ and $\gamma_i: \wt{\fM}^{\ev}_{\btau_i}\rightarrow \fM^{\ev}_{\btau_i}$ be the canonical maps from moduli spaces with evaluation logarithmic structures to basic log structures. Let $\beta_i: \wt{\fM}^{\ev}_{\btau_i,\red}\rightarrow \wt{\fM}^{\ev}_{\btau_i}$ be the canonical maps from the reduced induced stack $\wt{\fM}^{\ev}_{\btau_i,\red}$ to $\wt{\fM}^{\ev}_{\btau_i}$. Then, for the splitting morphism $\delta': \fM^{\ev}_{\btau} \rightarrow \prod_{i=1}^r \fM^{\ev}_{\btau_i}$ defined in \eqref{setupdiagram}, in the Chow group of $\prod_{i=1}^r \fM^{\ev}_{\btau_i}$, the following equation holds 
\begin{equation*}
    \delta'_*[\fM^{\ev}_{\btau}] = (\prod_{i=1}^r\gamma_i \circ \beta_i)_* \delta^{\ev}_{\red*}[\wt{\fM}^{\gl,\ev}_{\red}].
\end{equation*}
Here $\delta^{\ev}_{\red}: \wt{\fM}_{\red}^{\gl,\ev} \rightarrow \prod_{i=1}^r \wt{\fM}^{\ev}_{\btau_i,\red}$ is the map induced from $\delta^{\ev}: \wt{\fM}^{\gl,\ev} \rightarrow \prod_{i=1}^r \wt{\fM}^{\ev}_{\btau_i}$ in diagram \eqref{gluingfiberdiagram} by taking the reduction.
\end{proposition}

\begin{proof}
By \cite[Prop.5.5]{ACGSPunc}, which we recalled in Proposition \ref{isoafterenhanced}, the underlying stack morphism of $\gamma_{\btau}:\wt{\fM}^{\ev}_{\btau}\rightarrow \fM^{\ev}_{\btau}$ is an isomorphism. Then, by Lemma \ref{gluingreducedlemma}, the following diagram is commutative
\begin{equation*}
\begin{tikzcd}
\wt{\fM}^{\gl,\ev}_{\red} = \wt{\fM}^{\ev}_{\btau} \arrow[r,"\gamma_{\btau}"] \arrow[d,"\delta^{\ev}_{\red}"] & \fM^{\ev}_{\btau}
\arrow[d,"\delta'"] \\
\prod_{i=1}^r \wt{\fM}^{\ev}_{\btau_i,\red} \arrow[r,"\prod \gamma_i \circ \beta_i"] & \prod_{i=1}^r \fM^{\ev}_{\btau_i}.
\end{tikzcd}
\end{equation*}
Therefore,
\begin{equation*}
\begin{split}
\delta'_*[\fM^{\ev}_{\btau}] & = \delta'_* \gamma_{\btau*} [\wt{\fM}^{\gl,\ev}_{\red}] \\
& = (\prod_{i=1}^r \gamma_i\circ \beta_i)_*  \delta^{\ev}_{\red*} [\wt{\fM}^{\gl,\ev}_{\red}].
\end{split}
\end{equation*}

\end{proof}

In order to show Lemma \ref{gluingreducedlemma}, we need punctured maps \textit{weakly marked} by a global type $\btau$ defined in \cite[Def.3.7]{ACGSPunc}, and the moduli space of basic log punctured maps of weak marking by $\btau$, which is denoted $\wt{\fM}^{\ev'}_{\btau}$. In contrast to the moduli spaces of punctured maps marked by $\btau$, it carries an extra non-reducedness obtained from the infinitesimal deformation along $\btau$, which naturally occurs in the gluing process. See \cite[\S3.5.6]{ACGSPunc} for a more detailed discussion of moduli spaces of maps of weak marking. Here, we use this as a bridge between $\wt{\fM}_{\btau}$ and $\wt{\fM}^{\gl,\ev}$.

\begin{theorem}\cite[Cor.5.13]{ACGSPunc} \label{gluingformalism}
There is a fine, saturated fiber product of logarithmic stacks
\begin{equation*}
\begin{tikzcd} [row sep = small, column sep = small]
\wt{\fM}^{\ev'}_{\btau} \arrow[r] \arrow[d] & \wt{\fM}^{\ev'}_{\btau_1}\times_B\ldots\times_B \wt{\fM}^{\ev'}_{\btau_r} \arrow[d] \\
X_{\btau} \arrow[r,"\Delta"] & X_{\btau_1}\times_B\ldots\times_B X_{\btau_r}.
\end{tikzcd}
\end{equation*}
\end{theorem}

Following the fiber diagram 
\begin{equation*}
\begin{tikzcd}[row sep = small, column sep = small]
\wt{\fM}^{\ev'}_{\btau_1}\times_B\ldots\times_B \wt{\fM}^{\ev'}_{\btau_r} \arrow[r] \arrow[d] & \prod_{i=1}^r \wt{\fM}^{\ev'}_{\btau_i} \arrow[d] \\
X_{\btau_1}\times_B\ldots\times_B X_{\btau_r} \arrow[r] & \prod_{i=1}^r X_{\btau_i},\end{tikzcd}
\end{equation*}
with horizontal maps induced from the universal property of the fiber products, we have
\begin{equation} \label{gluingequation}
\wt{\fM}^{\ev'}_{\btau} = \prod_{i=1}^r\wt{\fM}_{\btau_i}^{\ev'} \times_{\prod_{i=1}^r X_{\btau_i}} X_{\btau}.
\end{equation}

\begin{proof}[\textbf{Proof of Lemma \ref{gluingreducedlemma}}]
It is shown in \cite[Prop 3.28]{ACGSPunc} that $\fM_{\btau}$ is reduced. The smoothness of the underlying stacks morphism of $\fM^{\ev}_{\btau}\rightarrow \fM_{\btau}$ induces that $\fM^{\ev}_{\btau}$ is reduced. As $\bS$ is a subset of edges, by Proposition \ref{isoafterenhanced}, the canonical map of  moduli spaces $\wt{\fM}_{\btau}^{\ev}\rightarrow \fM^{\ev}_{\btau}$ is an isomorphism on the underlying stacks. Therefore, the moduli space $\wt{\fM}^{\ev}_{\btau}$ is reduced. By \cite[Prop.3.31]{ACGSPunc}, there are closed embeddings $\wt{\fM}^{\ev}_{\btau} \rightarrow \wt{\fM}^{\ev'}_{\btau}$ and $\wt{\fM}^{\ev}_{\btau_i} \rightarrow \wt{\fM}^{\ev'}_{\btau_i}$ defined by nilpotent ideals. Hence,
\begin{equation*}
\wt{\fM}^{\ev}_{\btau}  = \wt{\fM}^{\ev'}_{\btau,\red}, \quad \wt{\fM}^{\ev}_{\btau_i,\red} = \wt{\fM}^{\ev'}_{\btau_i,\red}.
\end{equation*}
Then, by Theorem \ref{gluingformalism} and equation \eqref{gluingequation}, we obtain that
\begin{equation*}
\begin{split}
\wt{\fM}^{\ev}_{\btau} & = \wt{\fM}^{\ev'}_{\btau,\red}  = (\prod_{i=1}^r \wt{\fM}_{\btau_i}^{\ev'} \times_{\prod_{i=1}^r X_{\btau_i}} X_{\btau})_{\red}\\
    & = (\prod_{i=1}^r \wt{\fM}_{\btau_i}^{\ev} \times_{\prod_{i=1}^r X_{\btau_i}} X_{\btau})_{\red}  = \wt{\fM}^{\gl,\ev}_{\red}.
\end{split}
\end{equation*}

\end{proof}

\section{Generalization of Fulton-Sturmfels formula} \label{fssection}
The idealized log smoothness of the evaluation maps provides us with local toric models, where the local splitting maps can be seen as a toric morphism of toric stacks. We defer the discussion of the local toric models to the next Chapter. In this chapter, we study the pushforward of fundamental class under the morphisms of toric stacks. It is a generalization of the classical result of Fulton and Sturmfels on the intersection product of toric varieties. This chapter serves as a technical foundation for the splitting formula. The readers can feel free to skip the chapter first and check back later.

Let $X$ be a toric variety associated to a fan $(\Sigma(X),N(X))$. Let $N(Y)\subseteq N(X)$ be a saturated sublattice defining a subtorus $T_{Y}\subseteq T_{X}$. Define the scheme $Y$ to be the closure of $T_{Y}$ in $X$.

\begin{definition}\label{genericvector}
A vector $v\in N(X)$ is \textit{generic with respect to pairs $(X,Y)$} if for any cone $\delta \in \Sigma(X)$ with dimension $\dim X - \dim Y$, the affine space $N(Y)_{\R}+v$ intersects $\delta$ at at most one point, and if they intersect, the intersection point lies in the interior of $\delta$. 

Define $\Delta^0(v)$ to be the set of cones
\begin{equation*}
    \Delta^0(v):= \{\delta \in \Sigma(X)\mid \dim \delta = \dim X- \dim Y, (N(Y)_{\R}+v) \cap \delta \neq \varnothing \}.
\end{equation*}

\end{definition}

The Chow groups of a toric variety are generated by its toric strata. It is proved in \cite[Lemma 4.4]{FS97} that the subvariety $Y$ in $X$ is rationally equivalent to a linear combination of the toric strata determined by cones in $\Delta^0(v)$, for any generic displacement vector $v$ with respect to $(X,Y)$. Here, we provide a slightly different proof using the $\GG_m$-orbit of $Y$ under the torus action associated to a generic vector $v$.

\begin{lemma} \label{FS}
Let $v\in N(X)$ be a generic vector with respect to pairs $(X,Y)$. Then, in the Chow group $A_{\dim Y}(X)$,
\begin{equation} \label{FSoriginal}
[Y] = \sum_{\delta \in \Delta^0(v)} m(\delta) \cdot [V_{X}(\delta)].
\end{equation} 
Here $m(\delta) = [N(X): N_\delta+N(Y)]$, with $N_{\delta}$ the sublattice of $N(X)$ generated by the cone $\delta$. The subscheme $V_X(\delta)$ is the closed subvariety associated to cone $\delta$.
\end{lemma}
\begin{proof}
Let $L$ be the toric variety $X \times \mathbb{P}^1$ with the product fan structure. Let $\alpha: L \rightarrow \mathbb{P}^1$ be the projection and $\ol{\alpha}: N(L)\rightarrow \ZZ$ be the associated projection of lattices. 

We first construct the $\GG_m$-orbit of $Y$ under the torus action of $v$ as a subvariety of $X\times \PP^1$. Define $N(K)_{\R}:= \{(x+tv,t)\mid x\in N(Y)_{\R}$ and $t\in \R\}$ and $N(K)$ the saturated integral lattice $N(K)_{\R} \cap N(L)$. Let $T_K \subseteq T_L$ be the corresponding subtorus. The closure of $T_K$ defines a subvariety $K$ of $L$. By toric geometry, the preimage subscheme $\alpha|_{T_K}^{-1}(1)$ is isomorphic to $T_Y$ and $\alpha|_K^{-1}(1)$ is isomorphic to $Y$.

Let $\Sigma(K)$ be the fan with lattice $N(K)$ and cones $\delta \cap N(K)_{\RR}$ for $\delta \in \Sigma(L)$. Let $\widetilde{K}$ be the toric variety associated to $(\Sigma(K),N(K))$. Then $K$ is the image of $\widetilde{K}$ under the proper toric morphism associated to the lattice morphism $\beta_N: N(K) \hookrightarrow N(L)$. Let $\alpha': \widetilde{K} \xrightarrow{\beta} L \xrightarrow{\alpha} \mathbb{P}^1$ be the induced projection to $\mathbb{P}^1$.

By \cite[Prop 3.1]{abramovich2020decomposition}, the subscheme  $(\alpha')^{-1}({0})$ satisfies the equation
\begin{equation*}
    [(\alpha')^{-1}(0)] = \sum_{\tau} m_\tau \cdot [V_{\widetilde{K}}(\tau)],
\end{equation*} 
where $\tau$ goes over the rays in $\Sigma(K)$, whose image under the $\RR$-linear map 
\begin{equation*}
\ol{\alpha}'_{\RR}: N(K)_{\RR} \xrightarrow{\beta_{N,\RR}} N(L)_{\RR} \xrightarrow{\ol{\alpha}_{\RR}} \RR
\end{equation*}
is $\RR_{\geq 0}$. The multiplicity $m_\tau$ is given by the image of the primitive generator of $\tau$ under $\ol{\alpha}'_{\RR}$. Since $\alpha$ is a flat dominant morphism, by the alternative definition of rational equivalence in \cite[\S 1.6]{IntersectionTheory}, we obtain a rational equivalence relation in the Chow group of $X$:
\begin{equation} \label{FSmiddleequation}
\begin{split}
    [Y] = [\alpha|_K^{-1}(1)] \sim [\alpha|_K^{-1}(0)] & = \beta'_* [(\alpha')^{-1}(0)]\\
    & = \sum_{\tau} m_\tau \cdot \beta'_* [V_{\widetilde{K}}(\tau)].
\end{split}
\end{equation}
Here $\beta': \wt{K}\rightarrow X$ is composition of the map $\beta: \wt{K}\rightarrow L$ with the projection $L\rightarrow X$. It is sufficient to show that the above equation \eqref{FSmiddleequation} is the same as equation \eqref{FSoriginal}. 

First, there is an one-to-one correspondence between rays $\tau$ and cones in $\Delta^0(v)$.  
Note each ray $\tau\in \Sigma(K)$ is the intersection of $\delta \times \R_{\geq 0}$ and $N(K)_{\RR}$ for a cone $\delta$ in $\Sigma(X)$. The preimage $\ol{\alpha}_{\RR}^{-1}(1) \cap \tau$ is a point $(x+v,1)$ in  $\delta_{\RR} \times \R_{\geq 0}$ for some $x\in N(Y)_{\R}$. Hence $\{(N(Y)_{\R}+v) \cap \delta\}$ is non-empty. By the genericity of $v$, there is only one intersection point. It follows that the cone $\delta$ lies in $\Delta^0(v)$. On the other hand, for every $\delta$ in $\Delta^0(v)$, the intersection of $\delta \times \RR_{\geq 0}$ and $N(K)_{\RR}$ is a ray with image $\RR_{\geq 0}$ under $\ol{\alpha}'_{\RR}$.

Next, we need to show the corresponding multiplicities are the same. Under the proper morphism $\beta: \wt{K}\rightarrow K$, the image of $V_{\widetilde{K}}(\tau)$ is $V_L(\delta\times \RR_{\geq 0})$ up to a multiplicity. The multiplicity is decided by the degree of the finite map of the open toric strata from $O_{\widetilde{K}}(\tau)$ to $O_L(\delta\times \RR_{\geq 0})$, that is, the lattice index
\begin{equation*}
\begin{split}
& [N(L)/(N_{\delta} \times \ZZ): N(K)/N_\tau]\\  = & [N(L)/(N_{\delta}\times \ZZ): N(K)/((N_{\delta}\times \ZZ) \cap N(K))] \\  = & [N(L)/(N_{\delta}\times \ZZ):(N_{\delta}\times \ZZ+N(K))/(N_{\delta}\times \ZZ)]\\
= & [N(L):N_{\delta}\times \ZZ+N(K)].
\end{split}
\end{equation*}
Thus we have 
\begin{equation} \label{eq1}
   \begin{split}
       \beta'_* [V_{\widetilde{K}}(\tau)] =[N(L):N_\delta\times \ZZ + N(K)] \cdot [V_X(\delta)].
\end{split}
\end{equation}

The sublattice
\begin{equation*}
\begin{split}
    N_{\delta}\times \ZZ +N(K) & = N_{\delta}\times \ZZ + N(Y)\times \{0\} + \Z\cdot(v,1)\\
    & = N_\delta \times \Z + N(Y)\times \{0\} + \Z\cdot (v,0) \\
    & = (N_\delta+N(Y))\times \Z + \Z\cdot (v,0),
\end{split}
\end{equation*}
with the second and the third equality following from the fact that $(0,1) \in N_{\delta}\times \ZZ$.

Write $v = \frac{a_1}{a_2}\cdot x + \frac{b_1}{b_2} \cdot y$, where $x\in N_\delta$, $y\in N(Y)$ and integers pairs $(a_1,a_2)$, $(b_1,b_2)$ are coprime. As $N_{\delta}$ and $N(Y)$ has complementary dimension in $N(X)$, such presentation of $v$ is unique. The lattice index
\begin{equation} \label{eq2}
\begin{split}
    & [(N_\delta \times \ZZ+ N(K)): (N_\delta+N(Y))\times \Z] \\  = & [(N_\delta+N(Y))\times \Z + (\ZZ\cdot v,0): (N_\delta+N(Y))\times \Z]\\
     = & [N_\delta+N(Y)+(\ZZ\cdot v,0):N_\delta+N(Y)]  = \operatorname{lcm}(a_2,b_2),
\end{split}
\end{equation}
where $\operatorname{lcm}(a_2,b_2)$ is the least common multiple of integers $a_2$ and $b_2$. The integral generator $v_\tau$ of the ray $N(K)_{\RR}\cap \delta$ has form
\begin{equation*}
\begin{split}
    v_\tau & = n\cdot (v,1) + (y',0) \\
    & = \Big(\frac{n\cdot a_1}{a_2}\cdot x+ \frac{n\cdot b_1}{b_2} \cdot y+y',n \Big)
\end{split}
\end{equation*}
for $y'\in N(Y)$ and $n\in \Z$. Since $v_{\tau}$ is integral, then $n$ is a multiple of $\operatorname{lcm}(a_2,b_2)$. As $n$ is the smallest integer such that $v_{\tau}$ is integral, then $n = \operatorname{lcm}(a_2,b_2)$. Therefore
\begin{equation*}
m_\tau = \operatorname{lcm}(a_2,b_2) = [N_\delta \times \ZZ+ N(K): (N_\delta+N(Y))\times \Z] .
\end{equation*} 
The multiplicity
\begin{equation*}
    \begin{split}
 & m_\tau \cdot [N(L):(N_\delta\times \ZZ + N(K))]\\
       = & [N_\delta\times \ZZ +N(K): (N_\delta+N(Y))\times \Z] \cdot [N(L):N_\delta \times \ZZ+N(K)] \\
        = & [N(L):(N_\delta + N(Y))\times \Z] \\
         = & [N(X): N_\delta+N(Y)] = m_{\delta}.
    \end{split}
\end{equation*}
Then by \eqref{FSmiddleequation} and \eqref{eq1}, we obtain
\begin{equation*}
[Y] = \sum_{\tau} m_{\tau}\cdot \beta'_*[V_{\wt{K}}(\tau)] = \sum_{\delta\in \Delta^0(v)} m_{\delta} \cdot [V_X(\delta)].
\end{equation*}
\end{proof}

{\color{black}
\begin{example}
Let $m$ be a non-negative integer. Let $X$ be the Hirzebruch surface $F_m$, whose fan $\Sigma_X$ in $\ZZ^2$ contains four rays $r_1,...,r_4$ with directions $(1,0), (0,1), (-1,m)$ and $(0,-1)$. Let $N(Y)$ be one dimensional sublattice generated by $v_Y = (1,1)$. 

Suppose the generic displacement vector $v = (1,0)$, then $\Delta^0(v)$ contains rays $r_1$ and $r_4$. Since the lattice generated by $v_Y$ and $r_1$, $v_Y$ and $r_4$ are both $\ZZ^2$, the multiplicities for both rays are $1$. Suppose we take the generic displacement vector $v = (-1,0)$, then $\Delta^0(v)$ contains rays $r_2$ and $r_3$. The multiplicity for $r_2$ is $1$ and the multiplicity for $r_3$ is $m+1$. We obtain that in $A_1(X)$,
\begin{equation*}
    [Y]=[V_X(r_1)] + [V_X(r_4)] =[V_X(r_2)] + (m+1)\cdot [V_X(r_3)].
\end{equation*} 
\end{example}}

We generalize Lemma \ref{FS} to morphisms of toric strata. Let $\wt{f}: Y \rightarrow X$ be a proper morphism of toric varieties associated to an injective map of lattices $f_N: N(Y) \rightarrow N(X)$. Let $\tau$ be a cone in $\Sigma(Y)$ and let $\tau'$ be the smallest cone in $\Sigma(X)$ that contains the image of $\tau \in \Sigma(Y)$. Let $f: V_Y(\tau) \rightarrow V_X(\tau')$ be the restriction of $\wt{f}$ on $V_Y(\tau)$. Assuming 
\begin{equation*}
\dim \tau = \dim (\tau' \cap f_N(N(Y))),    
\end{equation*}
in which case the lattice map $N(Y) / N_{\tau} \rightarrow N(X)/N_{\tau'}$ associated to the toric morphism $f$ is injective. We wish to study $f_*[V_Y(\tau)]$ using the same idea.

\begin{definition} rder to  \label{genericdefinition2}
A vector $v\in N(X)$ is \textit{generic with respect to $(X,Y,V_Y(\tau))$} if its image under the quotient map $q_X: N(X)\rightarrow N(X)/N_{\tau'}$ is generic with respect to the pair $(V_X(\tau'), f(V_Y(\tau)))$ as defined in Definition \ref{genericvector}.

Similarly, we define $\Delta^{\tau}(v)$ to be the collection of cones $\delta$ in $\Sigma(X)$ satisfying that  
\begin{enumerate}
    \item $\tau' \subseteq \delta$,
    \item $\dim{N_\delta} = \dim N(X) - \dim N(Y) + \dim \tau$, and 
    \item $(f_N(N(Y))_{\RR}+v) \cap \delta$ is not empty.
\end{enumerate}
 
\end{definition}

\begin{proposition} \label{FSfortoricstrata}
In the Chow group $A_l(V_X(\tau'))$:
\begin{equation} \label{keyequation1}
    f_*[V_Y(\tau)] = \sum_{\delta \in \Delta^{\tau}(v)} m(\delta) \cdot [V_{X}(\delta)],
\end{equation}
where $l = {\dim N(Y)-\dim \tau}$ and $m(\delta) = [N(X): f_N(N(Y))+N_{\delta}]$. 
\end{proposition}

\begin{proof}
Let $q_X: N(X)\rightarrow N(X)/N_{\tau'}$ be the quotient of the lattice. Let $N'$ be the saturation of the image $q_X(f_N(N(Y)))$ in $N(X)/N_{\tau'}$. As the lattice map $N(Y) / N_{\tau} \rightarrow N(X)/N_{\tau'}$ is injective, the image of $V_Y(\tau)$ under $f$ is the closure of $T_{N'} \subseteq T_{N(X)/N_{\tau'}}$ inside $V_X(\tau')$. The degree of the map is the degree of the cover of torus induced from the saturation $q_X(f_N(N(Y))) \rightarrow N'$. Therefore,
\begin{equation} \label{aaequation}
\begin{split}
    f_*[V_Y(\tau)]  = [N':q_X(f_N(N(Y)))] \cdot [f(V_Y(\tau))].
\end{split}
\end{equation}

We first apply Lemma \ref{FS} to study $[f(V_Y(\tau))]$. By definition, the vector $q_X(v)$ is generic with respect to $(V_X(\tau'), f(V_Y(\tau)))$. Hence, in the Chow group of $V_X(\tau')$,
\begin{equation} \label{quotientFSequation}
    [f(V_Y(\tau))] = \sum_{\omega \in \Delta^0(q_X(v))} [N(X)/N_{\tau'}: N'+N_{\omega}] \cdot [V_{V_X(\tau')}(\omega)].
\end{equation}

Let us first show that a cone $\omega \in \Delta^0(q_X(v))$ if and only if $\delta \in \Delta^{\tau}(v)$ for $\delta$ the unique cone containing $\tau'$ and $q_X(\delta) = \omega$. Note that the intersection $\omega \cap (N'_{\RR}+q_X(v))$ is not empty if and only if the preimage of it under $q_X$ is not empty. That is, the intersection of $\delta+N_{\tau',\RR}$ with $f_N(N(Y))_{\RR}+N_{\tau',\RR}+v$ is not empty. It is equivalent to that the vector $v$ lies in $\delta+f_N(N(Y))_{\RR}+N_{\tau',\RR}$. We claim that 
\begin{equation}\label{latticeequa}
\begin{split}
\delta+f_N(N(Y))_{\RR}+N_{\tau',\RR} & = \delta+f_N(N(Y))_{\RR}.
\end{split}
\end{equation} 
{\color{black}
Let $w\in f_N(N(Y))_{\RR}$ such that $-w$ lies in the interior of $\tau'$. Then, for any $v\in N_{\tau',\RR}$, by taking an integer $I$ large enough, the vector $w': = -w+\frac{v}{I}$ lies in the interior of $\tau'$. We have $v = I\cdot w'+I \cdot w$. Hence, the vector $v$ lies in $\delta+f_N(N(Y))_{\RR}$, and the equality in equation \eqref{latticeequa} follows. Therefore, the condition $(3)$ of Definition \ref{genericdefinition2} that $(f_N(N(Y))_{\RR}+v) \cap \delta$ is not empty is equivalent to $\omega \cap (N_{\RR}'+q_X(v))$ being non-empty.}

For the dimension condition, as $\tau' \subseteq \delta$,
\begin{equation*}
\begin{split}
	\dim \omega = 
    \dim q_X(N_{\delta}) = \dim N_\delta - \dim N_{\tau'}.
\end{split}
\end{equation*}
Then 
\begin{equation*}
\begin{split}
\dim \omega = \dim q_X(N(X)) -\dim q_X(f_N(N(Y)))
\end{split}
\end{equation*}
if and only if
\begin{equation*}
\begin{split}
\dim N_{\delta} & = \dim N_{\tau'} + \dim q_X(N(X))- \dim q_X(f_N(N(Y)))\\
& = \dim N(X) - \dim q_X(f_N(N(Y))) \\
& = \dim N(X) - \dim f_N(N(Y)) + \dim (f_N(N(Y))\cap N_{\tau'})\\
& = \dim N(X) - \dim N(Y) + \dim \tau.
\end{split}
\end{equation*}
Hence $\omega \in \Delta^0(q_X(v))$ if and only if $\delta\in \Delta^{\tau}(v)$. 

Note that $V_X(\delta) = V_{V_X(\tau')}(\omega)$ by definition. The equation \eqref{quotientFSequation} is then equivalent to 
\begin{equation} \label{compareequation}
    [f(V_Y(\tau))] = \sum_{\delta \in \Delta^{\tau}(v)} [N(X)/N_{\tau'}: N'+q_X(N_\delta)] \cdot [V_X(\delta)].
\end{equation} Together with the equation \eqref{aaequation}, we obtain that
\begin{equation} \label{bbequation}
f_*[V_Y(\tau)]  = \sum_{\delta \in \Delta^{\tau}(v)}[N':q_X(f_N(N(Y)))][N(X)/N_{\tau'}: N'+q_X(N_\delta))] \cdot [V_X(\delta)].
\end{equation}

Note $N'$ and $q_X(N_{\delta})$ have complementary dimensions in the lattice $N(X)/N_{\tau'}$, and the intersection of $q_X(f_N(N(Y)))$ and $q_X(N_\delta)$ is zero dimensional. Thus,
\begin{equation} \label{somecalcu}
\begin{split}
[N'+q_X(N_{\delta}): q_X(f_N(N(Y)))+q_X(N_{\delta})] = [N':q_X(f_N(N(Y)))],
\end{split}
\end{equation}
as both equal the lattice order of $q_X(f_N(N(Y))) /[q_X(f_N(N(Y))) \cap q_X(N_{\delta})]$ as a sublattice of $N'/(N' \cap q_X(N_{\delta}))$.
Since the quotient lattice
\begin{equation*}
\begin{split}
q_X(f_N(N(Y))) + q_X(N_{\delta}) & = q_X(f_N(N(Y))+N_{\delta}) \\& = (f_N(N(Y))+N_{\delta})/N_{\tau'},
\end{split}
\end{equation*}
the lattice index in the equation \eqref{bbequation} satisfies that
\begin{equation*}
\begin{split}
     & \quad [N':q_X(f_N(N(Y)))] \cdot [N(X)/N_{\tau'}: N'+q_X(N_\delta))]\\
     & \overset{\eqref{somecalcu}}{=} 
     [N(X)/N_{\tau'}: q_X(f_N(N(Y)))+ q_X(N_\delta)] \\
     & \overset{}{=} [N(X)/N_{\tau'}: (f_N(N(Y))+N_{\delta})/N_{\tau'}]\\
     & = [N(X):f_N(N(Y))+N_{\delta}].
\end{split}
\end{equation*}
We now finish the proof of the equation \eqref{keyequation1}.

\end{proof}

\begin{corollary} \label{keyquotientfs}

With the same assumption in Proposition \ref{FSfortoricstrata}, let $N_Q$ be a sublattice of $N(Y)$. The subtorus $T_Q\subseteq T_Y$ induces a $T_Q$-action on $V_Y(\tau)$ and $V_X(\tau')$. The morphism $f: V_Y(\tau) \rightarrow V_X(\tau')$ is $T_Q$-equivariant. 

Let $f_Q:[V_Y(\tau)/T_Q]\rightarrow [V_X(\tau')/T_Q]$ be the induced map on the quotient stacks. Let $v$ be a generic displacement vector with respect to $(X,Y,V_Y(\tau))$ as defined in Definition \ref{genericdefinition2}. Then there is a closed substack of $[V_X(\tau')/T_Q]\times \PP^1$ which induces the rational equivalence of $[f_Q([V_Y(\tau)/T_Q])]$ and \begin{equation*}
\sum_{\delta \in \Delta^{\tau}(v)} \frac{m(\delta)}{[\im (q_X\circ f_N)^{\sat}:\im (q_X\circ f_N)]} \cdot [V_X(\delta)/T_Q],
\end{equation*}
where $\im(q_X\circ f_N)^{\sat}$ is the saturation of sublattice $\im(q_X\circ f_N)$ in $N(X)/N_{\tau'}$  and 
\begin{equation*}
m(\delta) = [N(X): f_N(N(Y))+N_{\delta}].
\end{equation*}

In the Chow group $A_l([V_X(\tau')/T_Q])$,
\begin{equation} \label{fsforstack}
    f_{Q*}[V_Y(\tau)/T_Q] = \sum_{\delta \in \Delta^{\tau}(v)} m(\delta) \cdot [V_{X}(\delta)/T_Q],
\end{equation}
where $l = \dim Y - \dim \tau - \dim N_Q$.

\end{corollary}

\begin{proof}
Let $N'$ be the saturation of $f_N(N(Y))$ in $N(X)$. Then $q_X(N')$ is saturated in $N(X)/N_{\tau'}$. We first show that 
\begin{equation} \label{hahaequation}
[f_Q([V_Y(\tau)/T_Q])] = \sum_{\delta \in \Delta^{\tau}(v)} [N(X)/N_{\tau'}:q_X(N')+q_X(N_{\delta})] \cdot [V_X(\delta)/T_Q]
\end{equation}
similar to the equation \eqref{compareequation} in the toric variety case.

In the toric subvariety $V_X(\tau')$, the closure of the torus associated to $N'$ is $f(V_Y(\tau))$. Let $v'$ be the vector $q_X(v)$ in $N(X)/N_{\tau'}$. By Lemma \ref{FS}, there is a closed subvariety $K$ in $V_X(\tau')\times \mathbb{P}^1$ defined to be the closure of the torus associated to the subspace 
\begin{equation*}
\{(x+t\cdot v',t) \mbox{ }|\mbox{ } x\in N(X)_{\RR}/N_{\tau',\RR} \mbox{ and } t\in \R\},
\end{equation*}
such that the projection map $\alpha: K \rightarrow \mathbb{P}^1$ induces the rational equivalence of $[f(V_Y(\tau))]$ and 
\begin{equation*}
\sum_{\delta \in \Delta^{\tau}(v)} [N(X)/N_{\tau'}:q_X(N')+q_X(N_{\delta})] \cdot [V_X(\delta)].
\end{equation*}
The inclusion of lattices 
\begin{equation*}
\begin{split}
N_Q \times \{0\}\hookrightarrow N_Q\times \ZZ \hookrightarrow N(Y)\times \ZZ \xrightarrow{f_N\times \id} N(X)\times \ZZ \xrightarrow{q_X\times \id} N(X)/N_{\tau'}\times \ZZ
\end{split}
\end{equation*}
induces a $T_Q$-action on $V_X(\tau')\times \PP^1$. An easy lattice computation tells us that $K$ is invariant under the $T_Q$-action and each fiber of $\alpha$ is $T_Q$-invariant. Therefore, the closed substack $[K/T_Q]$ together with the dominant morphism $\alpha': [K/T_Q] \rightarrow \mathbb{P}^1$ satisfies the equations
\begin{equation*}
    \begin{split}
        [\alpha'^{-1}(1)] & = [f(V_Y(\tau))/T_Q] = [f_Q([V_Y(\tau)/T_Q])], \\
        [\alpha'^{-1}(0)] & = \sum_{\delta \in \Delta^{\tau}(v)} [N(X)/N_{\tau'}:q_X(N')+q_X(N_{\delta})] \cdot [V_X(\delta)/T_Q].
    \end{split}
\end{equation*}
Hence this induces the rational equivalence of $[f(V_Y(\tau))/T_Q]$ and \begin{equation*}
\sum_{\delta \in \Delta^{\tau}(v)} [N(X)/N_{\tau'}:q_X(N')+q_X(N_{\delta})] \cdot [V_X(\delta)/T_Q].
\end{equation*}

In Proposition \ref{FSfortoricstrata}, we showed that the equation \eqref{compareequation} induces that
\begin{equation*}
    f_*[V_Y(\tau)] = \sum_{\delta \in \Delta^{\tau}(v)} m(\delta) \cdot [V_{X}(\delta)].
\end{equation*}
With the same argument, the equation \eqref{hahaequation} induces that
\begin{equation*}
\begin{split}
     f_{Q*}[V_Y(\tau)/T_Q] & = [N':f_N(N(Y))]\cdot [f_Q([V_Y(\tau)/T_Q])] \\&  = \sum_{\delta \in \Delta^{\tau}(v)} [N(X): f_N(N(Y))+N_{\delta}] \cdot [V_X(\delta)/T_Q].  
\end{split}
\end{equation*}
\end{proof}

\section{Proof of the splitting formula} \label{proofofthemaintheorem}

From now on, let us assume the Assumption \ref{toricassumption} is satisfied. 
\begin{repAssumption}{toricassumption}
Assume $B = \spec(\kk\rightarrow Q_B)$ is  a log point with $Q_B$ a toric monoid. Suppose $X\rightarrow B$ is an integral, log smooth morphism between fine, saturated log schemems. Assume $\ol{\cM}_X$ is globally generated, and for each edge $p\in\bS$, the strict closed subscheme $V_p$ of the log scheme $X$ has the underlying scheme a toric variety, and the log stratification of $V_p$ is the same as the toric stratification.  
\end{repAssumption}

{\color{black}
\subsection{Toric Strata Assumption} \label{toricstrata}
In the following two results, we show that the gluing strata $V_p$ in the Assumption \ref{toricassumption} are isomorphic to toric strata of toric varieties. 

\begin{proposition} \label{Vp proposition}
Suppose there is a log point $B = \spec(\kk\rightarrow Q_B)$ with $Q_B$ a toric monoid, and a log morphism $X\rightarrow B$ satisfying Assumption \ref{toricassumption}. Then, for each logarithmic stratum $V_p$, there is an idealized log structure $\cK_p$ on $V_p$, such that $V_p$ is idealized log smooth.
\end{proposition}
\begin{proof}
As $B$ is a log point, we define an idealized structure on $B$ with ideal $K_B = Q_B\backslash\{0\}$. It determines an idealized structure on $X$ such that $X\rightarrow B$ is ideally strict. Following Construction \ref{constructionidealizedconstruction}, we construct an idealized structure $\cK_{V_p}$ on $V_p$, determined by the pullback of $\cK_X$ and the ideal sheaf of $V_p$ in $X$. By \cite[Variant IV.3.1.21]{LogAlgebraicGeometry}, the idealized log scheme $V_p$ is idealized log \etale over log scheme $X$, hence is idealized log smooth over $B$. As $B$ is idealized log smooth, the log stratum $V_p$ is idealized log smooth.
\end{proof}

\begin{theorem} \label{toricstratacondition}
Suppose $X$ is a fine, saturated idealized log smooth scheme with Zariski log structure $\cM_X$ and idealized structure $\cK_X$. Assume $\ol{\cM}_X$ is globally generated. Suppose further that $X$ satisfies the following conditions:
\begin{enumerate}
\item The underlying scheme $\ul{X}$ is a toric variety.

\item The log stratification of $X$ is the same as the toric stratification of $X$. In other words, for each point $x$ of $X$, let $\wt{\sigma}$ be the dual cone of $\ol{\cM}_{X,x}$, the underlying scheme of the logarithmic stratum $\ul{V_X(\wt{\sigma})}$ is the smallest closed toric stratum of $\ul{X}$ containing $x$.

\end{enumerate} 

Then, there exists a toric variety $Y$ and a cone $\wt{\sigma}_0\in \Sigma_{Y}$, such that $X$ is isomorphic to $V_{Y}(\wt{\sigma}_0)$ as an idealized log scheme. Here, the idealized log structure on $V_{Y}(\wt{\sigma_0})$ is the idealized structure of $V_{Y}(\wt{\sigma_0}) \rightarrow Y$ following Construction \ref{constructionidealizedconstruction}, where $Y$ has trivial idealized structure.

\end{theorem}

\begin{proof}
Let $Q_0 = \ol{\cM}_X(T)$ for $T$ the maximal torus of $X$. Let $N$ be the cocharacter lattice and $M$ be the character lattice of $\ul{X}$. We first construct the fan of $Y$ in the lattice $\wt{N} = N\times Q_0^*$ by constructing a cone $\wt{\sigma}$ in $\wt{N}$ for each cone $\sigma$ in fan $\Sigma$ of $X$.

Since $\cM_X|_T$ is a constant sheaf by assumption $(2)$, there is an isomorphism $s: \cM_X(T) \rightarrow \kk^{\times }\oplus M \oplus Q_0$. For each cone $\sigma \in \Sigma$, the restriction map $\chi_{\sigma}: \cM_X(U(\sigma)) \rightarrow \cM_X(T)$ determines a map 
\begin{equation*}
    \phi_{\sigma}: \cM_X(U(\sigma)) \xrightarrow{\chi_{\sigma}} \cM_X(T) \xrightarrow{s} \kk^{\times}\oplus M\oplus Q_0 \xrightarrow{\pr} M\oplus Q_0,
\end{equation*}
where $U(\sigma) = \spec \kk[\sigma^{\vee}\cap M]$. Define $\wt{\sigma}$ in $\wt{N}$ to be the dual cone of the image monoid $\im(\phi_{\sigma})$. For the zero cone $0\in \Sigma$, $\wt{\sigma}_0$ is simply $\{0\}\times Q_0^{\vee}$. Although the isomorphism $s$ is not canonical, different $s$ differs by a morphism $Q_0\rightarrow M$, which results in a linear transformation of $N\times Q_0^*$ of determinant $1$.  

We now show that the collection of cones $\wt{\sigma}$ forms a fan in $\wt{N}$. In the commutative diagram
\begin{equation*}
\begin{tikzcd}
& \cM_X(U(\sigma)) \arrow[r,"\alpha_{\sigma}"] \arrow[d] & \kk[P_{\sigma}] \arrow[d, hook] \\
M\oplus Q_0 \arrow[r,"i"] & \cM_X(T) \cong \kk^{\times} \oplus M \oplus Q_0 \arrow[r,"\alpha_T"] & \kk[M],
\end{tikzcd}
\end{equation*}
the image of $(\im \phi_{\sigma})$ under $\alpha_T\circ i$ factors through $\kk[P_{\sigma}]$. Hence, the subgroup $(\im \phi_{\sigma})^{\times}$ is contained in $(\im \phi_{\sigma})\cap (P_{\sigma}^{\times} \oplus \{0\})$. Furthermore, as the map $\kk[P_{\sigma}] \hookrightarrow \kk[M]$ is an inclusion map, the image of $\cM_X(U(\sigma))$ in $\kk[M]$ contains monomials $P_{\sigma}^{\times}$. Therefore, image $(\im \phi_{\sigma})$ contains $P_{\sigma}^{\times} \oplus \{0\}$.  The subgroup $(\im \phi_{\sigma})^{\times}$ equals $(\im \phi_{\sigma}) \cap (P_{\sigma}^{\times} \oplus  \{0\})$. For cones $\sigma \subseteq \tau$, there is a commutative diagram
\begin{equation*}
\begin{tikzcd}[row sep = small, column sep = small]
&(\im\phi_{\tau}) \arrow[rr] \arrow[dd,dashed] & & (\im \phi_{\tau}) / (\im \phi_{\tau})^{\times} \arrow[dd] \\
\cM_X(U(\tau)) \arrow[rr] \arrow[dd]\arrow[ru] & & \ol{\cM}_X(U(\tau)) \arrow[ru] \arrow[dd] \\
& (\im\phi_{\sigma}) \arrow[rr,dashed] &&
(\im \phi_{\sigma}) / (\im \phi_{\sigma})^{\times}\\
\cM_X(U(\sigma)) \arrow[ur,dashed] \arrow[rr] & & \ol{\cM}_X(U(\sigma)) \arrow[ur]
\end{tikzcd}
\end{equation*}
where the horizontal maps of the right face are the maps induced from left face, by taking the quotient of the units.   As $(\im \phi_{\tau})$ is the image of $\cM_X(U(\tau))$ in $(\im \phi_{\sigma})$, we have $(\im \phi_{\tau}) / ((\im \phi_{\tau})\cap(\im \phi_{\sigma})^{\times})$ is the image of $\cM_X(U(\tau))$ in $(\im \phi_{\sigma})/ (\im \phi_{\sigma})^{\times}$, which is the image of 
\begin{equation} \label{strictimage}
\ol{\cM}_X(U(\tau)) \rightarrow \ol{\cM}_X(U(\sigma)) \rightarrow (\im \phi_{\sigma})/ (\im \phi_{\sigma})^{\times}.
\end{equation}
Since $\ol{\cM}_X$ is globally generated, the first map in \eqref{strictimage} is surjective. On the other hand, as the map $\cM_X(U(\sigma)) \rightarrow (\im \phi_{\sigma})$ is surjective by definition, the map to $(\im \phi_{\sigma}) / (\im \phi_{\sigma})^{\times}$ is surjective. Hence the second map in \eqref{strictimage} is also surjective. So is the whole map \eqref{strictimage}. Therefore, the map $(\im \phi_{\tau}) / (\im \phi_{\tau})^{\times}$ to $(\im \phi_{\sigma}) / (\im \phi_{\sigma})^{\times}$ is surjective, defined by taking the quotient by $(\im \phi_{\tau})\cap (P_{\sigma}^{\times} \oplus \{0\})$. Dually, we then obtain the face inclusions of the dual cones $\wt{\tau} \hookrightarrow \wt{\sigma}$. The collection of cones forms a fan $\wt{\Sigma}$ in $\wt{N}$.

Let $Y$ be the toric variety of $(\wt{\Sigma}, \wt{N})$. Note in the quotient fan in $\wt{N}/(\{0\}\times Q_0^*)$, the affine subvariety associated to $(\im \phi_{\sigma})^{\vee}$ is determined by the image of monoid
\begin{equation*}
    (\im \phi_{\sigma}) \rightarrow M\oplus Q_0 \xrightarrow{p} \kk[M],
\end{equation*}
where $p$ sends $(m,0)$ to $z^m$ and $(m,q)$ to $0$ for $q\neq 0$. Let $F_{\sigma}$ be the face $(\im \phi_{\sigma})\cap (M\oplus \{0\})$. We claim that $p(F_{\sigma})$ is isomorphic to $P_{\sigma} = \sigma^{\vee} \cap M$. As a consequence, the underlying variety $\ul{V_Y(\wt{\sigma}_0)}$ is isomorphic to $\ul{X}$. 

{
\color{black}{
First, as there is a commutative diagram
\begin{equation*}
\begin{tikzcd}
 \cM_X(U(\sigma)) \arrow[d,"\chi_{\sigma}"] \arrow[r,"\alpha_{\sigma}"]  & \kk[P_{\sigma}] \arrow[d,hook]\\
\cM_X(T)  \arrow[r,"\alpha_T"] & \kk[M],
\end{tikzcd}
\end{equation*}
the image of $f_{\sigma} = \alpha_T \circ \chi_{\sigma}$ lies in $\kk[P_{\sigma}]$. Since the image of $\alpha_T$ in $\kk[M]$ are monomials, the image of $f_{\sigma}$ lies in $(\kk^{\times} \oplus P_{\sigma}) \cup \{0\}$, hence $p(F_{\sigma})$ is a subset of $P_{\sigma}$. In order to prove that $p(F_{\sigma})$ is surjective on $P_{\sigma}$, we need a local description of log scheme $X$ in terms of monoids $\im(\phi_{\sigma})$. Before that, let us show that the morphism  
\begin{equation*}
\ol{\phi}_{\sigma} : \ol{\cM}_X(U(\sigma)) \rightarrow \im(\phi_{\sigma}) / \im(\phi_{\sigma})^{\times}
\end{equation*}
is an isomorphism of monoids. By definition of $\phi_{\sigma}$, the morphism is surjective. Suppose $p \in \ol{\cM}_X(U(\sigma))$ is mapped to $0 \in \im(\phi_{\sigma}) / \im (\phi_{\sigma})^{\times}$, with $z^p\in \cM_X(U(\sigma))$ a lift of $p$. As the units $\im(\phi_{\sigma})^{\times} = \im(\phi_{\sigma}) \cap (P_{\sigma}^{\times} \oplus\{0\})$, we have 
\begin{equation*}
\alpha_\sigma(z^p)\in \kk^{\times}\oplus P_{\sigma}^{\times},
\end{equation*}
hence $z^p$ lies in $\cM_X(U(\sigma))^{\times}$ and $p=0$. The morphism $\ol{\phi}_{\sigma}$ is an isomorphism of monoids. Note that $X$ is idealized log smooth, with log structure on $U(\sigma)$ determined by $\alpha_\sigma$. Then, smooth locally $\spec \kk[P_{\sigma}]$ is isomorphic to the product of a smooth scheme with
\begin{equation*} 
\spec \kk[\ol{\cM}_X(U(\sigma))]/(\ol{\cK}_X(U(\sigma))) = \spec \kk[(\im \phi_{\sigma})/(\im \phi_{\sigma})^{\times}]/(K_{\sigma}),
\end{equation*} where $K_{\sigma}$ is the image of $\ol{\cK}_X(U(\sigma))$ under $\ol{\phi}_{\sigma}$. The ideal $\ol{\cK}_X(U(\sigma))$ contains the elements of $\ol{\cM}_X(U(\sigma))$ which lift to elements in $\cM_X(U(\sigma))$ whose image under $\alpha_\sigma$ is zero. It follows that the ideal
$K_{\sigma}$ is generated by the image of $(\im \phi_{\sigma}) \backslash F_{\sigma}$ under the quotient of $(\im \phi_{\sigma})^{\times}$.
Hence, smooth locally $\spec \kk[P_{\sigma}]$ is isomorphic to the product of a smooth scheme with
\begin{equation*}
\spec \kk[F_{\sigma}/F_{\sigma}^{\times}]  \cong \spec \kk[F_{\sigma}/(P_{\sigma}^{\times} \oplus \{0\}) \cap F_{\sigma}] \cong \spec \kk[p(F_{\sigma})/P_{\sigma}^{\times}].
\end{equation*}
Furthermore, since each toric stratum of $U(\sigma)$ is a logarithmic stratum, each non-invertible element $g\in P_{\sigma}$ determines a monomial whose vanishing set is the same as the vanishing set of an element in $\im(f_{\sigma})$. It induces that $n\cdot g \in p(F_{\sigma})$ for some integer $n$. As $\spec \kk[P_{\sigma}]$ is isomorphic to the product of a smooth scheme with $\kk[p(F_{\sigma})/P_{\sigma}^{\times}]$, it is only possible if $n=1$. Therefore, $ p(F_{\sigma})$ is isomorphic to $P_{\sigma} = \sigma^{\vee} \cap M$ and $\ul{V_Y(\wt{\sigma}_0)}$ is isomorphic to $\ul{X}$.

We finish the proof by showing the idealized log structures of $V_Y(\wt{\sigma}_0)$ is the same as the idealized log structure on $X$. For each toric affine subvariety $U(\sigma)$ with morphisms
\begin{equation*}
\cM_X(U(\sigma)) \rightarrow \im(\phi_{\sigma}) \xrightarrow{p} \kk[P_{\sigma}],
\end{equation*}
the idealized log structure as a subvariety of $V_Y(\wt{\sigma}_0)$ is induced by the second map and ideal $K_{\sigma} = (\im \phi_{\sigma})\backslash F_{\sigma}$. The log structure from $X$ is induced by the whole map. Since $\ol{\cM}_X(U(\sigma)) \rightarrow \im(\phi_{\sigma}) / \im(\phi_{\sigma})^{\times}$ is an isomorphism, and both idealized structures are obtained by the preimages of $0$, two idealized log structures are isomorphic. We now finish the proof that two idealized log structures are the same. 
}}

\end{proof}
}

Following the above discussion, for each gluing stratum $V_p = V_Y(\sigma)$, there exists a toric variety $X_p$ such that $V_p$ is a strict toric stratum of $X_p$. Let $(\Sigma_p,N_p)$ be the fan of $X_p$ and $\delta_p$ be the cone with $V_p = V_{X_p}(\delta_p)$. The log map $V_p\rightarrow B$ induces a lattice map $N_p\rightarrow Q_B^*$ dual of 
\begin{equation*}
    Q^{\gp}_B = \ol{\cM}^{\gp}_B(B) \rightarrow \ol{\cM}^{\gp}_X(T_p) = Q^{\gp}_0
    \hookrightarrow M\oplus Q^{\gp}_0,
\end{equation*}
for $T_p$ the maximal torus of $X_p$. Hence, there is a toric morphism $X_p\rightarrow \spec \kk[Q_B]$. The map $V_p\rightarrow B$ is the restriction of $X_p\rightarrow \spec \kk[Q_B]$ on $X$.

Recall the definitions of $X_{\btau}$ and $X_{\btau_i}$ in \eqref{Xtaudefinition}
\begin{equation*}
\begin{split}
X_{\btau}: = V_{p_1}\times_B^{\operatorname{fs}}...\times^{\operatorname{fs}}_B V_{p_{|\bS|}},\quad p_j\in \bS\\
X_{\btau_i}:= V_{p_1}\times^{\operatorname{fs}}_B\ldots\times^{\operatorname{fs}}_B V_{p_{|\bS_i|}}, \quad p_j\in \bS_i.
\end{split}
\end{equation*} 
The next proposition studies the structure of $X_{\btau}$ and $X_{\btau_i}$. By \cite{MolchoSam2021Ussr}, the fine, saturated fiber product of toric varieties is determined by the fiber product of fans, which is defined in Appendix \ref{fiberproductoffans}. As $X_{\btau}$ and $X_{\btau_i}$ are fiber products of toric strata, they are the subschemes of the fiber product of toric varieties. Though the ideal determining $X_{\btau}$ and $X_{\btau_i}$, which is generated by the pullback ideals from $V_p$, might not be radical, the reduction of them are well understood in terms of toric strata.

\begin{proposition}\label{toriclocallemma}
The fiber product $X_{\btau}$ is a disjoint union of log schemes and each of them is isomorphic to 
an irreducible, but possibly non-reduced subscheme $Z_{\btau}$ of the toric variety $X'_{\btau}$ with fan
\begin{equation*}
\Sigma_{\btau} = \Sigma_{p_1}\times_{\Sigma(B)}\ldots\times_{\Sigma(B)} \Sigma_{p_{|\bS|}},\quad p_j\in \bS,
\end{equation*}
and with the toric log structure. The reduction of $Z_{\btau}$ is $V_{X'_{\btau}}(\delta)$, with $\delta$ the fiber product of cones $\delta_{p}$ for $p\in \bS$ over $\delta_B$. 
 
Similarly, for each $i = 1,2,\ldots,r$, the fiber product $X_{\btau_i}$
is a disjoint union of log schemes and each of them is isomorphic to 
an irreducible, but possibly non-reduced subscheme $Z_{\btau_i}$ of the toric variety $X'_{\btau_i}$ with fan \begin{equation*}
\Sigma_{\btau_i}  = \Sigma_{p_1}\times_{\Sigma(B)}\ldots\times_{\Sigma(B)} \Sigma_{p_{|\bS_i|}},\quad p_j\in \bS_i,
\end{equation*}
with the toric log structures. The reduction of $Z_{\btau_i}$ is $V_{X'_{\btau_i}}(\delta_i)$, with $\delta_i$ the fiber product of $\delta_p$ for $p\in \bS_i$ over $\delta_B$.

The fiber product $X'_{\btau} \times_{\prod_{i=1}^r X'_{\btau_i}} \prod_{i=1}^r Z_{\btau_i}$ is $Z_{\btau}$. With $Z_{\btau_i} \rightarrow X_{\btau_i}$ being the embedding of one component, the fiber product $X_{\btau} \times_{\prod_{i=1}^r X_{\btau_i}} \prod_{i=1}^r Z_{\btau_i}$ is a disjoint union of $\cN$ schemes, each of which is isomorphic to $Z_{\btau}$. Here, 
\begin{equation*}
\cN = [L^{\sat}: L], \quad L = \im(\ol{\Delta}_p)+ \prod_{i=1}^r \im(\ol{\Delta}_{\btau_i}),
\end{equation*}
where $\ol{\Delta}_p$ is the diagonal map of $\prod_{p\in \bS} N_p \rightarrow \prod_{p\in \bS} N_p\times N_p$ and $\ol{\Delta}_{\btau_i}$ is the lattice projection from $\prod_{i=1}^r \Sigma_{\btau_i}$ to $\prod_{p\in \bS} N_p\times N_p$.
\end{proposition}

\begin{proof}
By Lemma \ref{molcholemma}, the fine, saturated log fiber product 
\begin{equation} \label{fXprime}
X_{p_1}\times_{S_B}^{\operatorname{fs}}\ldots\times_{S_B}^{\operatorname{fs}}X_{p_{|\bS|}} ,\quad p_j \in \bS
\end{equation}
is a disjoint union of log schemes, each of which is isomorphic to the toric variety $X'_{\btau}$ of the fiber product of fans $(\Sigma_{\btau}, N_{\btau})$, with its toric log structure. Let $I_{p_i}$ be the ideal sheaf of $X_{p_i}$ that defines $V_{p_i}$. The scheme $X_{\btau}$ is then the subscheme of \eqref{fXprime} generated by the pullback of $I_{p_i}$. For toric morphisms $X'_{\btau} \rightarrow X_p$ and a cone $\delta\in \Sigma(X')$, the image of $V_{X'_{\btau}}(\delta)$ is contained in $V_p$ if and only if the image of $\delta$ under the fan map $N_{\btau}\rightarrow N_p$ intersects with the interior of $\delta_p$. Hence, the reduction of $Z_{\btau}$ is determined by the minimal cones $\delta$ with image intersecting with the interior of $\delta_p$. Let $\delta = \delta_{p_1}\times_{\delta_B}\ldots\times_{\delta_B} \delta_{p_{|\bS|}}$. The maps $\delta_p\rightarrow \delta_B$ are surjective, following the integrality of $X$ over $B$. Thus $\delta$ is mapped to the interior of $\delta_p$ under the projection map, and is the minimal cone satisfying the conditions. Therefore, the ideal $I$ defines an irreducible subscheme $Z_{\btau}$ whose reduction is the toric strata $V_{X'_{\btau}}(\delta)$. The proof works the same for $\btau_i$.

The subscheme $Z_{\btau}$ and the fiber product $X'_{\btau} \times_{\prod_{i=1}^r X'_{\btau_i}} \prod_{i=1}^r Z_{\btau_i}$  are both the subschemes of $X_{\btau}'$ determined by the pullback of ideals $I_p$ under $X_{\btau}'\rightarrow \prod_p X_p$. Hence they are the same. For the last statement, by the following lemma \ref{Xtautoproductlemma}, the equation
\begin{equation*}
\begin{split}
X_{\btau}\times_{\big(\prod_{i=1}^r X_{\btau_i}\big)} \big(\prod_{i=1}^r Z_{\btau_i}\big)  = \prod_{p\in \bS}X_p\times_{\big(\prod_{p\in \bS}X_p\times X_p\big)} \big(\prod_{i=1}^r Z_{\btau_i}\big)
\end{split}
\end{equation*}
holds. By Lemma \ref{molcholemma}, the right side is the union of $Z_{\btau}$ with the number of the components $\cN$ being the lattice index 
\begin{equation*}
\begin{split}
\cN & = 
[(\im(\ol{\Delta}_p)+ \prod_{i=1}^r \im(\ol{\Delta}_{\btau_i}))^{\sat}: \im(\ol{\Delta}_p)+ \prod_{i=1}^r \im(\ol{\Delta}_{\btau_i})].
\end{split}
\end{equation*}

\end{proof}

\begin{lemma} \label{Xtautoproductlemma}
Assume the graph $G$ of $\btau$ is connected. There is a Cartesian diagram in the category of fine, saturated logarithmic schemes
\begin{equation*}
\begin{tikzcd}[row sep = small, column sep = small]
X_{\btau} \arrow[r] \arrow[d] & \prod_{i=1}^r X_{\btau_i} \arrow[d] \\
\prod_{p\in \bS} X_p \arrow[r,"\Delta_{p}"] & \prod_{p\in \bS} X_p\times X_p, 
\end{tikzcd}
\end{equation*}
with horizontal maps the diagonal maps and the vertical maps the composition of the projections 
\begin{equation*}
g_{\btau}: X_{\btau} \rightarrow \prod_{p\in \bS} V_p, \quad g_{\btau_i}: X_{\btau_i} \rightarrow \prod_{p\in \bS_i} V_p
\end{equation*}
with the closed embeddings $V_p \hookrightarrow X_p$.
\end{lemma}

\begin{proof}
As $V_p \hookrightarrow X_p$ is a strict closed embedding, it is sufficient to show that the following diagram is Cartesian in the category of fine, saturated log schemes
\begin{equation} \label{inductiondiagram}
\begin{tikzcd}
X_{\btau} \arrow[r,"\Delta_X"] \arrow[d,"g_{\btau}"] & \prod_{i=1}^r X_{\btau_i} \arrow[d,"\prod g_{\btau_i}"] \\
\prod_{p\in \bS} V_p \arrow[r,"\Delta_p"] & \prod_{p\in \bS} V_p\times V_p.
\end{tikzcd}
\end{equation}
Let $Z$ be a fine, saturated log scheme with $\alpha: Z\rightarrow \prod_{i=1}^r X_{\btau_i}$ and $\beta: Z\rightarrow \prod_{p\in \bS} V_p$, such that $\prod_{i=1}^r g_{\btau_i}\circ \alpha = \Delta_p \circ \beta$ as logarithmic maps. Then, there is a commutative diagram
\begin{equation*}
\begin{tikzcd}[row sep = small, column sep = small]
Z \arrow[r,"\alpha"] \arrow[d,"\beta"] & \prod_{i=1}^r X_{\btau_i} \arrow[d,"\prod g_{\btau_i}"] \arrow[r] & \prod_{i=1}^r B \arrow[dd]\\
\prod_{p\in \bS} V_p \arrow[d] \arrow[r,"\Delta_p"] & \prod_{p\in \bS} V_p\times V_p \arrow[rd]\\
\prod_{p\in \bS} B \arrow[rr] & & \prod_{p\in \bS}  B\times B.
\end{tikzcd}
\end{equation*}
By the universal property of the fiber products, there is a morphism
\begin{equation*}
Z\rightarrow \big(\prod_{i=1}^r B\big) \times_{\big(\prod_{p\in \bS} B\times B\big)} \big(\prod_{p\in \bS} B\big ).
\end{equation*}
As the dual graph $G$ of $\btau$ is connected, the pullback of $\prod_{i=1}^r B$ along the diagonal map identifies the base of each $X_{\btau_i}$. Thus
$\big(\prod_{i=1}^r B\big) \times_{\big(\prod_{p\in \bS} B\times B\big)} \big(\prod_{p\in \bS} B\big ) = B$, with the maps from $B$ to each factor being diagonal maps.  For each $p \in \bS$, the projection $Z\rightarrow V_p\rightarrow B$ is the same as the map 
\begin{equation*}
Z\rightarrow B \xrightarrow{\Delta_B}\prod_{p\in \bS} B \xrightarrow{\pr_p} B.
\end{equation*}
By the universal property of the logarithmic fiber products, there is a unique morphism 
\begin{equation*}
\psi: Z\rightarrow X_{\btau} = V_{p_1}\times_B\ldots\times_B V_{p_{|\bS|}},
\end{equation*}
with $g_{\btau}\circ \psi = \beta$. On the other hand, both $\alpha$ and $\Delta_X \circ \psi$ are the unique morphism induced from the universal property of the fiber product
\begin{equation*}
\prod_{i=1}^r X_{\btau_i} = \big (\prod_{i=1}^r B\big) \times_{\big (\prod_{p\in \bS} B\times B\big)} \big(\prod_{p\in \bS} V_p\times V_p\big).
\end{equation*}
Hence $\Delta_X\circ \psi = \alpha$ and we finish the proof of the diagram \eqref{inductiondiagram} being Cartesian in the category of log fine, saturated schemes.

\end{proof}

\subsection{Local Toric Models of the Gluing Formalism}\label{localmodelsubsection}
We are now ready to study the gluing formalism under the Assumption \ref{toricassumption}. We first study the local structure of the splitting morphism
\begin{equation} \label{delevsplit}
\delta_{\red}^{\ev}: \wt{\fM}_{\red}^{\gl,\ev} \rightarrow \prod_{i=1}^r\wt{\fM}^{\ev}_{\btau_i,\red}. 
\end{equation}
The main result of this section is Proposition \ref{localsplittingformula}, which provides a local splitting equation \eqref{localsplittingequation} of a geometric point after a base change to an \etale neighborhood. The idea is to analyze $\delta^{\ev}_{\red}$ under the following commutative diagram obtained from the fiber diagram \eqref{gluingfiberdiagram}
\begin{equation*} 
\begin{tikzcd}
\wt{\fM}^{\gl,\ev}_{\red} \arrow[r,"\delta_{\red}^{\ev}"] \arrow[d,"\ev"] & \prod_{i=1}^r \wt{\fM}_{\btau_i,\red}^{\ev} \arrow[d, "\prod \ev_{\tau_i}"] \\ X_{\btau} \arrow[r,"\Delta_X"] & \prod_{i=1}^r X_{\btau_i}.
\end{tikzcd}
\end{equation*}

Let us first construct the \etale base change for a geometric point $\ol{x}$ on $\wt{\fM}^{\gl,\ev}_{\red}$. Let \begin{equation*}
Q_{\ol{x}} = \ol{\cM}_{\prod \wt{\fM}_{\btau_i,\red}^{\ev},\delta_{\red}^{\ev}(\ol{x})}, \quad L_{\ol{x}} = \ol{\cK}_{\prod\wt{\fM}_{\btau_i,\red}^{\ev},\delta_{\red}^{\ev}(\ol{x})}.
\end{equation*}
By \cite[Appendix B.2]{ACGSPunc}, there is a connected strict \etale neighborhood $U_{\ol{x}}$ of the geometric point $\delta^{\ev}_{\red}(\ol{x})$ in $\prod_{i=1}^r \wt{\fM}_{\btau_i,\red}^{\ev}$ such that there is a commutative diagram
\begin{equation} \label{Uxdefinition}
    \begin{tikzcd}[row sep = small, column sep = small]
    \prod_{i=1}^r \wt{\fM}_{\btau_i,\red}^{\ev} \arrow[d,"\prod{\ev_{\btau_i}}"] &    U_{\ol{x}} \arrow[r] \arrow[d,"\ev_U"] \arrow[l] & \A_{Q_{\ol{x}},L_{\ol{x}}}\arrow[d]\\
    \prod_{i=1}^r X_{\btau_i} &
\prod_{i=1}^r Z_{\btau_i} \arrow[l]\arrow[r] & \prod_{i=1}^r \A_{{\btau_i}},
    \end{tikzcd}
\end{equation}
with $\cA_{\tau_i}$ the Artin fan of the toric variety $X'_{\btau_i}$. Define
\begin{equation} \label{Ugluedefinition}
U_{\ol{x}}^{\gl} : = U_{\ol{x}}\times_{\prod_{i=1}^r \wt{\fM}_{\btau_i,\red}^{\ev}}
\wt{\fM}_{\red}^{\gl,\ev}.
\end{equation}
We wish to study the map $\delta_U^{\ev}: U^{\gl}_{\ol{x}} \rightarrow U_{\ol{x}}$.

\begin{proposition} \label{localsplittingformula}
Let $\fV$, $\Delta(\fV)$ and $m_{[\brho]}$ be the global gluing data associated to the splitting morphism  $\delta:\scrM(X/B,\btau) \rightarrow \prod_{i=1}^r \scrM(X/B,\btau_i)$ defined in Definition \ref{globalgluingdata}. Let $\delta_U^{\ev}: U^{\gl}_{\ol{x}} \rightarrow U_{\ol{x}}$ be an \etale local model of \eqref{delevsplit} at a geometric point $\ol{x}$ of $\wt{\fM}_{\red}^{\gl,\ev}$, defined by the above diagrams \eqref{Uxdefinition} and \eqref{Ugluedefinition}. Let
\begin{equation*}
    U_{\ol{x}}^{\brho} = U_{\ol{x}} \times_{\prod_{i=1}^r \wt{\fM}_{\btau_i,\red}^{\ev}} \prod_{i=1}^r \wt{\fM}^{\ev}_{\brho_i,\btau_i,\red},
\end{equation*}
with $\wt{\fM}^{\ev}_{\brho_i,\btau_i,\red}$ the image substack of the finite morphism $\wt{j}_{\brho_i,\btau_i}: \wt{\fM}^{\ev}_{\brho_i,\red}\rightarrow \wt{\fM}^{\ev}_{\btau_i,\red}$. 

Then, in the Chow group of $U_{\ol{x}}$,
\begin{equation} \tag{$\star$} \label{localsplittingequation}
    \delta_{U*}^{\ev}([U_{\ol{x}}^{\gl}]) = \sum_{[\brho] \in \Delta(\fV)}m_{[\brho]} \cdot [U^{\brho}_{\ol{x}}].
\end{equation}
\end{proposition}

\begin{proof}
Note that $\wt{\fM}^{\gl,\ev}_{\red}$ is the reduction of $\prod_{i=1}^r \wt{\fM}^{\ev}_{\btau_i,\red} \times_{\prod_{i=1}^r X_{\btau_i}} X_{\btau}$, so $U_{\ol{x}}^{\gl}$ is the reduction of $U_{\ol{x}}\times_{\prod_{i=1}^r X_{\btau_i}} X_{\btau}$. As $U_{\ol{x}}$ is connected, the evaluation map $\ev_U$ factors through one component of $\prod_{i=1}^r Z_{\btau_i}$. Hence,
\begin{equation*}
\begin{split}
 U_{\ol{x}}^{\gl} = [U_{\ol{x}}\times_{\prod_{i=1}^r X_{\btau_i}} X_{\btau}]_{\red} & = [U_{\ol{x}}\times_{\prod_{i=1}^r Z_{\btau_i}} (\prod_{i=1}^r Z_{\btau_i}\times_{\prod_{i=1}^r X_{\btau_i}} X_{\btau})]_{\red}.  
\end{split}
\end{equation*}
By Proposition \ref{toriclocallemma}, there is a lattice index $\cN$, such that $\prod_{i=1}^r Z_{\btau_i}\times_{\prod_{i=1}^r X_{\btau_i}} X_{\btau}$ is a disjoint union of $\cN$ schemes, each of which is isomorphic to $Z_{\btau}$. Hence, $U_{\ol{x}}^{\gl}$ is the disjoint union of $\cN$ schemes, each of which is isomorphic to the reduction of 
\begin{equation} \label{glirfiber}
    U^{\gl,\operatorname{ir}}_{\ol{x}}: = U_{\ol{x}}\times_{\prod_{i=1}^r Z_{\btau_i}} Z_{\btau} = U_{\ol{x}}\times_{\prod_{i=1}^r X'_{\btau_i}} X'_{\btau},
\end{equation}
where the equality follows from Proposition \ref{toriclocallemma}. Hence, it is sufficient to study the diagonal map $\delta'_U: U_{\ol{x}}^{\gl,\ir}\rightarrow U_{\ol{x}}$. 

First, we observe that the evaluation map $U_{\ol{x}} \rightarrow \prod_{i=1}^r X'_{\btau_i}$ is idealized log smooth. It is sufficient to show that $\wt{\fM}^{\ev}_{\btau_i,\red} \rightarrow X'_{\btau_i}$ is idealized log smooth. By Corollary \ref{evisidealizedlogsmooth}, $\wt{\fM}_{\btau_i,\red}$ is idealized log smooth over $B$. Locally, the map $\wt{\fM}_{\btau_i,\red} \rightarrow B$ factors through 
\begin{equation*}
    \wt{\fM}_{\btau_i,\red} \xrightarrow{\ev_{\btau_i}} \cZ_{\btau_i} \hookrightarrow \cX_{\btau_i}' \rightarrow B,
\end{equation*}
with $\cZ_{\btau_i}$ being the relative Artin fan $\cA_{Z_{\btau_i}} \times_{\cA_B} B$ and $\cX'_{\btau_i}$ being $\cA_{\btau_i} \times_{\cA_{B}} B$. Note $\X'_{\btau_i}$ is logarithmically \etale over $B$, hence the evaluation map $\wt{\fM}_{\btau_i,\red} \rightarrow \cX'_{\btau_i}$ is idealized log smooth. The map from the evaluation enhancement
\begin{equation*}
    \wt{\fM}_{\btau_i,\red}^{\ev} = \wt{\fM}_{\btau_i,\red}\times_{\cX_{\btau_i}} X_{\btau_i} = \wt{\fM}_{\btau_i,\red}\times_{\cZ_{\btau_i}} Z_{\btau_i} = \wt{\fM}_{\btau_i,\red}\times_{\cX'_{\btau_i}} X'_{\btau_i}
\end{equation*}
to $X'_{\btau_i}$ is then idealized log smooth. 

Next, we study the local splitting morphism using the toric local model of idealized log smooth morphisms. Following \eqref{Uxdefinition}, let us define
\begin{equation} \label{Aevdefinition}
\begin{split}
    \A_{\ol{x}}^{\ev}  = \A_{Q_{\ol{x}},L_{\ol{x}}} \times_{\prod_{i=1}^r \cA_{\btau_i}} \prod_{i=1}^r X'_{\btau_i} \quad \mbox{ and } \quad
    \A_{\ol{x}}^{\gl,\ev}  = \A_{\ol{x}}^{\ev}\times_{\prod_{i=1}^r X'_{\btau_i}}X'_{\btau}.
\end{split}
\end{equation} 
In the fine, saturated Cartesian diagram 
\begin{equation} \label{hugediagram}
\begin{tikzcd}[column sep = small, row sep = small]
U_{\ol{x}}^{\gl,\operatorname{ir}} \arrow[r,"\phi"] \arrow[d,"\delta_U'"] & \A_{\ol{x}}^{\gl, \ev} \arrow[d,"\delta_{\cA}^{\ev}"] \arrow[r] & X'_{\btau} \arrow[d]\\
U_{\ol{x}} \arrow[r,"\iota"]& \A^{\ev}_{\ol{x}} \arrow[r] \arrow[d] & \prod_{i=1}^r X'_{\btau_i} \arrow[d] \\
& \cA_{Q_{\ol{x}},L_{\ol{x}}}\arrow[r] & \prod_{i=1}^r \cA_{\btau_i},
\end{tikzcd}
\end{equation}
the map $\iota: U_{\ol{x}} \rightarrow \cA_{\ol{x}}^{\ev}$ is induced by the universal property of the fiber product. Since the evaluation map $U_{\ol{x}} \rightarrow \prod_{i=1}^r X_{\btau_i}'$ is idealized log smooth, by \cite[Appendix B.4]{ACGSPunc}, the map $\iota$ is smooth. As the diagonal map $X'_{\btau} \rightarrow \prod_{i=1}^r X'_{\btau_i}$ is proper, both  vertical maps $\delta_U'$ and $\delta_{\cA}^{\ev}$ in diagram \eqref{hugediagram} are proper. We obtain that 
\begin{equation} \label{relations}
\begin{split}
    \delta'_{U*}[U_{\ol{x}}^{\gl,\operatorname{ir}}] = \delta'_{U*}\phi^*[\A^{\gl, \ev}_{
    \ol{x}}] = \iota^*\delta^{\ev}_{\A*}[\A^{\gl, \ev}_{
    \ol{x}}]
\end{split}
\end{equation}
for $\phi: U_{\ol{x}}^{\gl,\operatorname{ir}} \rightarrow \cA_{\ol{x}}^{\gl,\ev}$ in the diagram \eqref{hugediagram}. 

It is sufficient to study $\delta^{\ev}_{\A*}[\A^{\gl, \ev}_{\ol{x}}]$ using the concrete toric stack description of $\cA_{\ol{x}}^{\gl,\ev}$ and $\cA_{\ol{x}}^{\ev}$.  Let $[\bomega] = (\bomega_1,\ldots,\bomega_r)$ be the global type of $\delta^{\ev}_{\red}(\ol{x})$. For $[\brho] = (\brho_1,\ldots,\brho_r)$ a global type that admits a contraction to $[\btau] = (\btau_1,\ldots,\btau_r)$, we define 
\begin{equation*}
    \A_{\ol{x}}^{\ev,\brho}: = \begin{cases} V_{\A_{\ol{x}}^{\ev}}(\wt{\brho}), &  \mbox{if $[\brho]$ is a contraction of $[\bomega]$,}\\
    \varnothing, & \mbox{ otherwise},
    \end{cases}
\end{equation*}
where $\wt{\brho} = \prod_{i=1}^r \wt{\brho}_i$ is the evaluation cone associated to $[\brho]$ defined in Definition \ref{evaluationcone}. By Lemma \ref{evaluationconetype} and Corollary \ref{stratatificationofmmm}, the evaluation cone $\wt{\brho}$ is a subcone of $\wt{\bomega} = Q_{\ol{x}}^{\vee}$, hence the stratum $V_{\cA^{\ev}_{\ol{x}}}(\wt{\brho})$ is well-defined. The following lemma provides a description of $\delta^{\ev}_{\A*}[\A^{\gl, \ev}_{\ol{x}}]$, whose proof we defer to later.

\begin{lemma} \label{pushforwardforAglev}
Let $\fV$ be a generic displacement as defined in Definition \ref{globalgluingdata}. Then, in the Chow group of $\A_{\ol{x}}^{\ev}$,
\begin{equation}
\label{Aevlemmaeq}          \delta^{\ev}_{\A*}[\A^{\gl, \ev}_{
    \ol{x}}] = \sum_{\brho\in \Delta(\fV)} m_{[\brho]}' \cdot [\A_{\ol{x}}^{\ev,\brho}],
\end{equation}
where 
\begin{equation*} 
m_{[\brho]}' = [\prod_{i=1}^r N_{\btau_i}: \im(\ol{\Delta}) + \prod_{i=1}^r \im(\ol{\ev}_{\brho_i})],
\end{equation*}
Here $\ol{\Delta}: N_{\btau} \rightarrow \prod_{i=1}^r N_{\btau_i}$ is
the lattice diagonal map of $X_{\btau}'\rightarrow \prod_{i=1}^r X_{\btau_i}'$. The map $\ol{\ev}_{\brho_i}: N_{\wt{\brho}_i}\rightarrow N_{\btau_i}$ is defined by the factorization of the tropical evaluation maps discussed in \eqref{evdiscussion}
\begin{equation*} 
\ol{\evt}_{\brho_i}: N_{\wt{\brho}_i} \xrightarrow{\ol{\ev}_{\brho_i}} N_{\btau_i} \rightarrow \prod_{p\in \bS_i} N_p.
\end{equation*}
\end{lemma}

At last, we are ready to finish the proof by the following arguments. By Lemma \ref{pushforwardforAglev} and equation \eqref{relations}, in the Chow group of $U_{\ol{x}}$, 
\begin{equation*}
\begin{split}
    \delta^{'}_{U*}[U^{\gl,\ir}_{
    \ol{x}}] & = \sum_{[\brho] \in \Delta(\fV)} m'_{[\brho]} \cdot [\iota^{-1}(\A_{\ol{x}}^{\ev,\brho})] = \sum_{[\brho]\in \Delta(\fV)} m'_{[\brho]} \cdot [U^{\brho}_{\ol{x}}].
\end{split}
\end{equation*}
The second equality follows from Corollary \ref{stratatificationofmmm}. Hence, following the discussion in the proof of Proposition \ref{localsplittingformula},
\begin{equation*}
\begin{split}
    \delta_{U*}^{\ev}([U_{\ol{x}}^{\gl}]) = \sum_{[\brho] \in \Delta(\fV)}\cN \cdot m'_{[\brho]} \cdot [U^{\brho}_{\ol{x}}].
    \end{split}
\end{equation*}

It is now sufficient to show $\cN \cdot m_{[\brho]}' = m_{[\brho]}$. 
All the involved lattices can be fit into a commutative diagram
\begin{equation} \label{importantlatticediagram}
\begin{tikzcd} 
N_{\wt{\brho}}\times_{\prod_{i=1}^r N_{\btau_i}} N_{\btau} \arrow[r] \arrow[d] & N_{\btau} \arrow[r] \arrow[d,"\ol{\Delta}"] & \prod_{p\in \bS} N_p  \arrow[d,"\ol{\Delta}_p"] \\
N_{\wt{\brho}} = \prod_{i=1}^r N_{\wt{\brho}_i} \arrow[r,"\prod\ol{\ev}_{\brho_i}"] \arrow[d, "\coker"]&
\prod_{i=1}^r N_{\btau_i} \arrow[r,"\prod_{i=1}^r \ol{\Delta}_{\btau_i}"] \arrow[d, "\coker \ol{\Delta}"] & \prod_{p\in \bS} N_p\times N_p \arrow[d, "\coker\ol{\Delta}_p"] \\
C_{\wt{\brho}} \arrow[r] & C_{\btau} \arrow[r] & C_{p}.
\end{tikzcd}
\end{equation} 
Note that $\cN = [N_{\cN}^{\sat}: N_{\cN}]$ and $m_{[\brho]}' = [\prod_{i=1}^r N_{\btau_i}: N_m]$, where 
\begin{equation*}
 N_{\cN} =  \im(\ol{\Delta}_p)+ \prod_{i=1}^r \im(\ol{\Delta}_{\btau_i}), \quad N_m = \im(\ol{\Delta}) + \prod_{i=1}^r \im(\ol{\ev}_{\brho_i}).
\end{equation*}
In other words, in diagram \eqref{importantlatticediagram}, the order $\cN$ is obtained by taking the saturation of the image lattice of two maps in the right corner of the upper right square, while the order $m'_{[\brho]}$ comes from the two maps in the upper left square. Let $K = \im(\ol{\Delta}_p) + \prod_{i=1}^r \im(\ol{\Delta}_{\btau_i} \circ \ol{\ev}_{\brho_i})$ to be the image lattice from the right corner of the big upper rectangle. We claim that 
\begin{equation} \label{indexeseqality}
\begin{split}
\cN \cdot m'_{[\brho]} & =  [N_{\cN}^{\sat} : N_{\cN}] \cdot [\prod_{i=1}^r N_{\btau_i} : N_m]\\
& = [(\coker \ol{\Delta}_p(K))^{\sat}:\coker \ol{\Delta}_p(K)]  \\& = \big[\im(\prod_{i=1}^r \ol{\varepsilon}_{\brho_i})^{\sat}:\im(\prod_{i=1}^r \ol{\varepsilon}_{\brho_i})] = m_{[\brho]},
\end{split}
\end{equation}
with $\ol{\varepsilon}_{\brho_i}$ is defined in equation \eqref{varepsilondefinition} and Definition \ref{globalgluingdata}, where the first, the third and the fourth equalities follow directly from the definition of $\cN$, $m_{[\brho]}'$, $\ol{\varepsilon}_{\brho_i}$ and $m_{[\brho]}$. 

{\color{black}
In order to prove the second equality, we will utilize following Proposition \ref{Prop1}. Before that, let us first show 
\begin{equation} \label{twoequalities}
(\coker \ol{\Delta}_p(K))^{\sat} = \coker  \ol{\Delta}_p(K^{\sat}) = \coker  \ol{\Delta}_p(N_{\cN}^{\sat}).
\end{equation}
Since $K = \im(\ol{\Delta}_p) + \prod_{i=1}^r \im(\ol{\Delta}_{\btau_i} \circ \ol{\ev}_{\brho_i})$, the image $\im (\ol{\Delta}_p)$ is a sublattice of $K$, hence
\begin{equation*}
\begin{split}
(\coker \ol{\Delta}_p(K))^{\sat} = (K/\im\ol{\Delta}_p)^{\sat}.
\end{split}
\end{equation*}
Since lattice $K/\im\ol{\Delta}_p$ is a sublattice of $K^{\sat}/\im\ol{\Delta}_p$, and both lattices has dimension $\dim K - \dim \im(\ol{\Delta}_p)$, the saturation of $K/\im\ol{\Delta}_p$ in $K^{\sat}/\im\ol{\Delta}_p$ is the full lattice $K^{\sat}/\im\ol{\Delta}_p$. Therefore, their saturations in the quotient lattice $C_p$ in \eqref{importantlatticediagram} are the same, that is,  
\begin{equation*}
(\coker \ol{\Delta}_p(K))^{\sat} = (K/\im\ol{\Delta}_p) ^{\sat} = (K^{\sat}/\im\ol{\Delta}_p)^{\sat}.
\end{equation*}
Suppose for $\alpha \in \prod_{p\in\cS}N_p\times N_p$ and $[\alpha] \in C_p$, there is an integer $m$ such that $m\cdot[\alpha]$ lies in $K^{\sat}/\im\ol{\Delta}_p$. As $\im \ol{\Delta}_p$ is a subset of $K$, then the set $m \cdot \alpha + \im\ol{\Delta}_p$ is a subset of $K^{\sat}$. Hence, $m\cdot \alpha \in K^{\sat}$ and $\alpha \in K^{\sat}$. It follows that $[\alpha] \in K^{\sat}/\im\ol{\Delta}_p$ and $K^{\sat} / \im \ol{\Delta}_p$ is saturated. The first equality of \eqref{twoequalities} is proved. We now prove the second equliaty. Following the definition of $K$ and $N_{\cN}$, it is obvious that $K \subseteq N_{\cN}$, hence $K^{\sat} \subseteq N_{\cN}^{\sat}$. On the other hand, because 
\begin{equation*}
\begin{split}
K & = \im \ol{\Delta}_p +  \im\prod_{i=1}^r(\ol{\Delta}_{\btau_i} \circ \ol{\ev}_{\brho_i}) \\& = \im \ol{\Delta}_p +  \im\prod_{i=1}^r(\ol{\Delta}_{\btau_i} \circ \ol{\ev}_{\brho_i}) +\im(\prod_{i=1}^r \ol{\Delta}_{\btau_i} \circ \ol{\Delta}) \\ & = \im \ol{\Delta}_p + \im\{\prod_{i=1}^r \ol{\Delta}_{\btau_i}\circ  (\ol{\Delta} +\prod_{i=1}^r\ol{\ev}_{\brho_i})\}, 
\end{split}
\end{equation*}
where the second equality follows as $\im(\prod_{i=1}^r \ol{\Delta}_{\btau_i} \circ \ol{\Delta})$ is a subset of $\im \ol{\Delta}_p$. Following Lemma \ref{pushforwardforAglev} and the dimension argument of its proof, the image $\im(\ol{\Delta} +\prod_{i=1}^r\ol{\ev}_{\brho_i})$ is a full dimensional sublattice of $\prod_{i=1}^r N_{\btau_i}$, hence $K$ has the same dimension as $N_{\cN}$, which induces that their saturation are the same. We now finish the proof of equation \eqref{twoequalities}. 

Now we use the following proposition to prove \eqref{indexeseqality}. We postpone the proof of proposition \ref{Prop1} to the last of the section.

We call a homomorphism $\alpha: A\arr B$ of lattices \emph{of finite index} if $\im(\alpha)\subseteq B$ is of finite index. In this case, we define the \emph{index of $\alpha$} by
\[
\ind(\alpha)= \big[B:\im(\alpha)\big] = \big|B/\im(\alpha)\big|.
\]
\begin{proposition}
\label{Prop1}
Let
\[
\xymatrix{
A'\ar[r]\ar[d]&B'\ar[r]\ar[d]^\beta&C'\ar[d]^\delta\\
A\ar[r]^\alpha&B\ar[r]^\gamma&C
}
\]
be a commutative diagram of lattices with $\gamma$ injective, and the right-hand square is
Cartesian. Let $q:C\arr \ol C$ be the
cokernel of $\delta$. Suppose the homomorphisms $q\circ\gamma\circ\alpha$, $\alpha+\beta$ and $\gamma+\delta$ are of finite index. Then it holds that
\[
\ind(q\circ\gamma\circ\alpha)=\ind(\alpha+\beta)\cdot\ind(\gamma+\delta).
\]
\end{proposition}

Following the diagram \eqref{importantlatticediagram}, we have 
\begin{equation} \label{smallimportdiagram}
\begin{tikzcd} 
N_{\wt{\brho}}\times_{\prod_{i=1}^r N_{\btau_i}} N_{\btau} \arrow[r] \arrow[d] & N_{\btau} \arrow[r] \arrow[d,"\ol{\Delta}"] & \prod_{p\in \bS} N_p  \arrow[d,"\ol{\Delta}_p"] \\
N_{\wt{\brho}} = \prod_{i=1}^r N_{\wt{\brho}_i} \arrow[r,"\prod\ol{\ev}_{\brho_i}"] &
\prod_{i=1}^r N_{\btau_i} \arrow[r,"\prod_{i=1}^r \ol{\Delta}_{\btau_i}"]  & N^{\sat}_{\cN}.
\end{tikzcd}
\end{equation} 
Then by equation \eqref{twoequalities} and Proposition \ref{Prop1},
\begin{equation*}
\begin{split}
 [(\coker \ol{\Delta}_p(K))^{\sat}:\coker \ol{\Delta}_p(K)]
& = [\coker\ol{\Delta}_p(N_{\cN}^{\sat}): \coker \ol{\Delta}_p(K)] \\ & = [N_{\cN}^{\sat} : N_{\cN}] \cdot [\prod_{i=1}^r N_{\btau_i} : N_m]  
\end{split}
\end{equation*}
We finish the proof of \eqref{indexeseqality} and hence Proposition \ref{localsplittingformula}.
}
\end{proof}

\begin{proof}[\textbf{Proof of Lemma \ref{pushforwardforAglev}}] \label{proofdetail}
First, we show that $\delta_{\A}^{\ev}$ is the map of quotient stacks induced from a $T_{Q_{\ol{x}}}$-equivariant map of toric varieties. By Proposition \ref{cartesian},
\begin{equation*}
\begin{split}
\cA^{\ev}_{\ol{x}} =  \A_{Q_{\ol{x}},L_{\ol{x}}} \times_{\prod_{i=1}^r \cA_{\btau_i}} \prod_{i=1}^r X'_{\btau_i} 
 = [S_{Q_{\ol{x}},L_{\ol{x}}} \times \prod_{i=1}^r T_{\btau_i} /T_{Q_{\ol{x}}}].
\end{split}
\end{equation*}
Let $Y$ be $S_{Q_{\ol{x}},L_{\ol{x}}}\times \prod_{i=1}^r T_{\btau_i}$ and let $Y^{\gl}: = Y\times_{\prod_{i=1}^r X'_{\btau_i}} X'_{\btau}$, where the map $Y\rightarrow \prod_{i=1}^r X_{\btau_i}'$ is the composition of the quotient map $Y\rightarrow \cA_{\ol{x}}^{\ev}$ and projection $\cA_{\ol{x}}^{\ev}\rightarrow \prod_{i=1}^r X_{\btau_i}$. By Lemma \ref{quotientstacklemma}, the torus action of $T_{Q_{\ol{x}}}$ on $Y$ induces an action of $T_{Q_{\ol{x}}}$ on $Y^{\gl}$, such that the quotient stack $[Y^{\gl}/T_{Q_{\ol{x}}}]$ is isomorphic to $\cA^{\gl,\ev}_{\ol{x}}$. 
The diagonal map $\delta^{\ev}_{\cA}: \cA_{\ol{x}}^{\gl,\ev} \rightarrow \cA_{\ol{x}}^{\ev}$
is then induced from the quotient of the $T_{Q_{\ol{x}}}$-equivariant map $Y^{\gl} \rightarrow Y$. In order to study $\delta^{\ev}_{\cA*}[\cA_{\ol{x}}^{\gl,\ev}]$, it is sufficient to study its reduction $\delta^{\ev}_{\cA*}[\cA_{\ol{x},\red}^{\gl,\ev}]$, hence it is enough to study $Y_{\red}^{\gl} \rightarrow Y$. 

Next, we show that $Y^{\gl}_{\red} \rightarrow Y$ is the restriction of a toric morphism on a toric stratum. Such description allows us to use the generalized Fulton-Sturmfels formula in Proposition \ref{keyquotientfs} to obtain Lemma \ref{pushforwardforAglev}. Let $\cY$ be $S_{Q_{\ol{x}}} \times \prod_{i=1}^r  {T_{\btau_i}}$. The tropicalization of  $\cA_{\ol{x}}^{\ev} \rightarrow \prod_{i=1}^r X'_{\btau_i}$ induces a toric morphism $\ev_{\cY}: \cY\rightarrow \prod_{i=1}^r X_{\btau_i}'$. By Lemma \ref{molcholemma}, the fine, saturated fiber product 
$\cY\times_{\prod_{i=1}^r X_{\btau_i}'} X'_{\btau}$ is the disjoint union of $n$ varieties, each of which is isomorphic to a toric variety $\cY^{\gl}$, with \begin{equation} \label{pushout}
n = [\prod_{i=1}^r N_{\btau_i}: \ol{\Delta}(N_{\btau})+ \ol{\ev}_{\cY}(N(\cY))],
\end{equation}
for $\ol{\Delta}$ defined in \eqref{importantlatticediagram}. By Corollary \ref{IdealFanLift}, the lattice map $\ol{\ev}_{\cY}$ is defined as
\begin{equation} \label{evlattice}
\begin{split}
\ol{\ev}_{\cY}: N(\cY) = N_{Q_{\ol{x}}}\times \prod_{i=1}^r N_{\btau_i} \rightarrow \prod_{i=1}^r N_{\btau_i},\quad (a,b) \mapsto e(a)-b,
\end{split}
\end{equation}
with $e$ the lattice map of the evaluation map $\cA_{Q_{\ol{x}},L_{\ol{x}}} \rightarrow \prod_{i=1}^r \cA_{\btau_i}$. The map $\ol{\ev}_{\cY}$ is surjective, hence $n=1$. 

The scheme \begin{equation*}
Y^{\gl} = Y \times_{\prod_{i=1}^r X'_{\btau_i}} X'_{\btau} = Y\times_{\cY} \cY^{\gl}
\end{equation*}
is the subscheme of $\cY^{\gl}$ determined by the pullback of the ideal generated by $L_{\ol{x}}$ in $S_{Q_{\ol{x}}}$. In particular, we claim that the subscheme $Y^{\gl}_{\red}$ is the toric stratum $V_{\cY^{\gl}}(\sigma^{\gl})$, for $\sigma^{\gl} = \wt{\btau} \times \{0\}$. By the idealized structure on $\wt{\fM}_{\btau,\red}^{\ev}$ in Remark \ref{Idealstructureremark}, $S_{Q_{\ol{x}},L_{\ol{x}}}$ is $V_{S_{Q_{\ol{x}}}}(\prod_{i=1}^r \wt{\btau}_i)$.
Therefore $Y$ is the toric strata $V_{\cY}(\prod_{i=1}^r \wt{\btau}_i \times \{0\})$. We use $\sigma_Y$ to denote $\prod_{i=1}^r \wt{\btau}_i \times \{0\}$. Since $\btau$ is realizable and the types $\btau_i$ are obtained by splitting $\btau$, by the construction of the evaluation cones in Definition \ref{evaluationcone}, $\sigma^{\gl} = \wt{\btau} \times \{0\}$ is exactly the fiber product of cones 
\begin{equation*}
\sigma_Y \times_{\prod_{i=1}^r \Sigma(X'_{\btau_i})} \Sigma(X'_{\btau}),
\end{equation*}
from the fiber product of toric varieties $\cY \times_{\prod_{i=1}^r X'_{\btau_i}} X'_{\btau}$. Then, the cone $\sigma^{\gl}$ is the unique minimal cone in $\Sigma(\cY^{\gl})$ whose image in $\Sigma(\cY)$ intersects the interior of $\sigma_Y$. The reduction $Y^{\gl}_{\red}$ is the toric stratum $V_{\cY^{\gl}}(\sigma^{\gl})$.

Now, we are ready to use the generalized Fulton-Sturmfels formula to study the diagonal map $\delta^{\ev}_{\cA}: \cA_{\ol{x},\red}^{\gl,\ev}\rightarrow \cA^{\ev}_{\ol{x}}$, which is the toric morphism of toric stacks
\begin{equation*}
\Big[V_{\cY^{\gl}}(\sigma^{\gl})/T_{Q_{\ol{x}}}\Big] \rightarrow \Big[V_{\cY}(\sigma_Y)/T_{Q_{\ol{x}}}\Big].
\end{equation*} 

Let $\fV \in \prod_{p\in \bS} N_p$ be a generic displacement vector defined in Definition \ref{globalgluingdata}. Let $\psi$ be the map
\begin{equation*}
    \psi: \prod_{i=1}^r N_{\btau_i}\rightarrow \prod_{p\in \bS} N_p\times N_p\rightarrow \prod_{p\in \bS} N_p,
\end{equation*}
whose first map is the projections of fiber products $N_{\btau_i}$ to $N_p$ and the second map is the cokernel of the diagonal map of $N_p$. We first show that there exists an element $v\in N(\cY)$ such that $\psi\circ \ol{\ev}_{\cY}(v) = \fV$. Let $N'$ be the sublattice $\prod_{p\in \bS}^B N_{p}$ of $\prod_{p\in \bS} N_p$, whose image in $N_B$ under maps $\prod_{p\in \bS} N_p\rightarrow N_p\rightarrow N_B$ are the same for any $p\in \bS$. By definition, vector $\fV$ lies in $N'$. Note that 
\begin{equation*}
    \psi(N_{\btau_1}\times_{N_B}\ldots\times_{N_B} N_{\btau_r}) = N'.
\end{equation*}
Hence $\psi^{-1}(\fV)$ is non-empty. Since $\ol{\ev}_{\cY}$ is surjective, we obtain that there exists $v$ in $N(\cY)$ such that $\psi\circ \ol{\ev}_{\cY}(v) = \fV$. Next, we want to show that $v$ is a generic displacement vector associated to $(\cY,\cY^{\gl},Y^{\gl}_{\red})$ as defined in Definition \ref{genericdefinition2}. The equation \eqref{Aevlemmaeq} can be obtained using Fulton-Sturmfels formula in Corollary \ref{keyquotientfs} associated to $v$.

\begin{enumerate}
    \item \textit{The vector $v$ is generic with respect to $(\cY,\cY^{\gl},Y^{\gl}_{\red})$.}
Let $f: N(\cY^{\gl}) \rightarrow N(\cY)$ be the lattice map and let $q_{\cY}: N(\cY)\rightarrow N(\cY)/N_{\sigma_Y}$ be the lattice quotient map. We need to show for any cone $\omega \in \Sigma(\cY)$ satisfying the conditions $(1),(2),(3)$ in Definition \ref{genericdefinition2}, the intersection $q_{\cY,\RR}((\im(f)+v) \cap \omega)$ lies in the interior of the cone $q_{\cY,\RR}(\omega)$. Following conditions $(1),(3)$, $\omega$ determines a unique type $[\brho] = (\brho_1,\ldots,\brho_r)$ with property $(i),(ii)$ in Definition \ref{globalgluingdata} $(2)$. Furthermore, the dimension of $\omega$ satisfies the dimension condition in property $(iii)$ since 
\begin{equation*}
\begin{split}
    & \sum_{i=1}^r \dim \wt{\brho}_i - \dim \wt{\btau}\\ = & \dim \omega - \dim \sigma^{\gl}\\  \overset{Def \ref{genericdefinition2}}{=}
    & \dim N(\cY) - \dim N(\cY^{\gl}) +\dim \sigma^{\gl} - \dim \sigma^{\gl} \\
    {=} & \dim N(\cY) - \dim \im(f) = \sum_{i=1}^r N_{\btau_i} - \dim N_{\btau} \\= & \sum_{p\in \bS}\dim N_p - (|\bS|-r+1) \cdot \rank Q_B^{\gp}.
    \end{split}
\end{equation*}

Recall that $\varepsilon_{\brho}$ is defined in \eqref{varepsilondefinition}
\begin{equation*}
\varepsilon_{\brho}: \prod_{i=1}^r\wt{\brho}_{i} \xrightarrow{\prod \evt_{\brho}} \prod_{p\in \bS} N_{p,\RR}\times N_{p,\RR} \xrightarrow{\prod \coker \ol{\Delta}_p} \prod_{p\in \bS} N_{p,\RR}.
\end{equation*}
By Remark \ref{equivalentRemark}, vector $\fV$ does not lie in the boundary of the image of $\varepsilon_{\brho}$. Then the preimage $(\varepsilon_{\brho})^{-1}(\fV)$ intersects with the interior of $\prod_{i=1}^r\wt{\brho}_{i}$. Note that the cone $\omega =\prod_{i=1}^r\wt{\brho}_i\times \{0\}.$ Hence, $(\ol{\ev}_{\cY})^{-1}\psi^{-1}(\fV)$ intersects with the interior of cone $\omega$. As $\im(f)$ contains the kernel of $\psi\circ\ol{\ev}_{\cY}$, the set $(\ol{\ev}_{\cY})^{-1}\psi^{-1}(\fV)$ is contained in $\im(f)+v$. Hence, the intersection of $\im(f) + v$ with $\omega$ is not empty and is in the interior of $\omega$. It follows that the intersection $q_{\cY,\RR}((\im(f)+v) \cap \omega)$ lies in the interior of the cone $q_{\cY,\RR}(\omega)$.

\item \textit{Lemma \eqref{pushforwardforAglev} follows from Fulton-Sturmfels Formula in Corollary \ref{keyquotientfs} on $v$.} By Corollary \ref{keyquotientfs}, we have
\begin{equation*}
\begin{split}
    & \delta_{\cA*}^{\ev}\Big[V_{\cY^{\gl}}(\sigma^{\gl})/T_{Q_{\ol{x}}}\Big] \\ = &\sum_{\rho'\in \Delta^{\sigma^{\gl}}(v)} [N(\cY): \im(f)+N_{\rho'}] \cdot [V_{\cY}(\rho')/T_{Q_{\ol{x}}}],
\end{split}
\end{equation*}
with $\Delta^{\sigma^{\gl}}(v)$ the set of cones in $N(\cY)$ defined in Definition \ref{genericdefinition2}. There is a bijection between types in $\Delta(\fV)$ and cones in $\Delta^{\sigma^{\gl}}(v)$, by taking a type $[\brho]$ to $\prod_{i=1}^r \wt{\brho}_i\times \{0\}$. The substack $\cA_{\ol{x}}^{\ev,\brho}$ is the same as $[V_{\cY}(\rho')/T_{Q_{\ol{x}}}]$ following the stratification of the moduli space in Corollary \ref{stratatificationofmmm}. As for multiplicities,
 since $\ol{\ev}_{\cY}$ is surjective and the kernel is contained in $\im(f)$, by taking the quotient of $\ker(\ol{\ev}_{\cY})$, we get
\begin{equation*} 
\begin{split}
[N(\cY):\im(f)+N_{\rho'}] & = [\prod_{i=1}^r N_{\btau_i}: \ol{\ev}_{\cY}(\im(f) + N_{\rho'})]
\\ & = [\prod_{i=1}^r N_{\btau_i}: \im(\ol{\Delta}) + \prod_{i=1}^r \im(\ol{\ev}_{\brho_i})] = m_{[\brho]}'.
\end{split}
\end{equation*}
\end{enumerate}
\end{proof}

\begin{remark} \label{compatibilityremark}
With the same assumption in Lemma \ref{pushforwardforAglev}, by Corollary \ref{keyquotientfs}, there is a closed substack $K_{\A} \subseteq \A_{\ol{x}}^{\ev} \times \PP^1$ and a projection map $\alpha_{\cA}: K_\A \rightarrow \PP^1$ such that as algebraic cycles in $\A_{\ol{x}}^{\ev}$,
\begin{equation*}
\begin{split}
    [\alpha_{\cA}^{-1}(1)]   = [\delta_{\A}^{\ev}(\A_{\ol{x},\red}^{\gl,\ev})] ,\quad 
    [\alpha_{\cA}^{-1}(0)]   = \sum_{[\brho]\in \Delta(\fV)} \frac{m'_{[\brho]}}{\fI_{\ol{x}}} \cdot [{\A_{\ol{x}}^{\ev,\brho}}],
\end{split}
\end{equation*}
with $\fI_{\ol{x}} = [\im(q_{\cY}\circ f)^{\operatorname{sat}}:\im(q_{\cY}\circ f)]$ 
for $N(\cY^{\gl})\xrightarrow{f} N(\cY) \xrightarrow{q_{\cY}} N(\cY)/N_{\sigma_Y}$ defined as above. The index $\fI_{\ol{x}}$ is the degree of map $\delta_{\cA}^{\ev}$.

Similarly, take $K_{U_{\ol{x}}}$ to be the preimage of $K_{\cA}$ under the smooth map $\iota \times \id$ from $U_{\ol{x}} \times \PP^1$ to  $\cA_{\ol{x}}^{\ev} \times \PP^1$. Let $\alpha_U: K_{U_{\ol{x}}} \rightarrow \PP^1$ be the projection map. As the diagram \eqref{hugediagram} is fine, saturated Cartesian, as algebraic cycles in $U_{\ol{x}}$,
\begin{equation}\label{induceformulacondition}
\begin{split}
    [\alpha_U^{-1}(1)] & = [\delta^{\ev}_U(U^{\gl}_{\ol{x}})], \quad
    [\alpha_U^{-1}(0)]  = \sum_{[\brho]\in \Delta(\fV)} \frac{m_{[\brho]}}{\fI_{\ol{x}}} \cdot [U_{\ol{x}}^{\brho}].
\end{split}
\end{equation}The closed substack $K_{U_{\ol{x}}}$ induces the local splitting equation \eqref{localsplittingequation} in Proposition \ref{localsplittingformula}.
\end{remark}

{\color{black}{
At last, we finish the proof of Proposition \ref{Prop1}.
\begin{proof}[\textbf{Proof of Proposition \ref{Prop1}}]
\label{proofProp1}
The diagram induces a sequence
\[
\xymatrix{
\ol A\ar[r]^{\ol \alpha}& \ol B\ar[r]^{\ol \gamma}&\ol C
}
\]
of cokernels of the vertical maps. Note that $\ol\gamma$ is injective as
$\gamma$ is injective and the right-hand square is Cartesian. We have
\begin{equation*}
\ind(q\circ \gamma \circ \alpha) = [\ol{C} : \im (q\circ \gamma \circ \alpha)] = [\ol{C} : \im (\ol{\gamma} \circ \ol{\alpha})] = \ind(\ol\gamma\circ\ol\alpha).
\end{equation*}
Applying the following
Lemma~\ref{Lem2} and then twice Lemma~\ref{Lem3} below, we obtain the finite
index properties and the claimed equality:
\[
\ind(q\circ\gamma\circ\alpha)=
\ind(\ol\gamma\circ\ol\alpha)=\ind(\ol\alpha)\cdot\ind(\ol\gamma)=
\ind(\alpha+\beta)\cdot\ind(\gamma+\delta).
\]
\end{proof}

\begin{lemma}
\label{Lem2}
Let
\[
\xymatrix{
A\ar[r]^\alpha& B\ar[r]^\beta&C
}
\]
be a sequence of lattices with $\alpha,\beta$ of finite index and $\beta$ injective. Then $\beta\circ\alpha$ is of finite index, and it holds
\[
\ind(\beta\circ\alpha)=\ind(\beta)\cdot\ind(\alpha).
\] 
\end{lemma}

\begin{proof}
Notice first that by replacing $A$ by $A/\ker(\alpha)$ we may assume $\alpha$ is
also injective. Now consider the following diagram with exact rows and columns:
\[
\xymatrix{
&0\ar[d]&0\ar[d]&0\ar[d]\\
0\ar[r]&A\ar[r]^\alpha\ar[d]&B\ar[r]\ar[d]^\beta&Q_1\ar[r]\ar[d]&0\\
0\ar[r]&\beta(\alpha(A))\ar[r]\ar[d]&C\ar[r]\ar[d]&Q\ar[r]\ar[d]&0\\
&0\ar[r]&Q_2\ar[r]\ar[d]&Q_2\ar[r]\ar[d]&0\\
&&0&0
}
\]
Here we first fill in $Q_1$ and $Q$ by completing the second and third lines, and then complete the diagram using the Snake Lemma.

The statement now follows from
\[
\ind(\beta\circ\alpha)=\big[C:\beta(\alpha(A))\big] =\big|Q\big|=
\big|Q_1\big|\cdot\big|Q_2\big|=\ind(\alpha)\cdot\ind(\beta).
\]
\end{proof}

\begin{lemma}
\label{Lem3}
Let
\[
\xymatrix{
A'\ar[r]^\gamma\ar[d]^{\alpha'}&A\ar[r]\ar[d]^\alpha&\ol A\ar[r]\ar[d]^{\ol\alpha}&0\\
B'\ar[r]^\beta&B\ar[r]&\ol B\ar[r]&0
}
\]
be a commutative diagram of lattices with exact rows. {\color{black}{Suppose $\alpha + \beta$ is a
homomorphism of finite index.}} Then $\ol\alpha$ is of finite index and $\ind(\ol\alpha)=\ind(\alpha+\beta)$.
\end{lemma}

\begin{proof}

We consider the following commutative diagram with exact rows and columns:
\[
\xymatrix{
A'\oplus B'\ar[r]^{\gamma\oplus\id}\ar[d]_{\alpha'+\id}&
A\oplus B' \ar[r]\ar[d]^{\alpha+\beta}& \ol A\ar[d]^{\ol\alpha}\ar[r]&0\\
B'\ar[r]^\beta&B\ar[r]\ar[d]& \ol B\ar[r]\ar[d]&0\\
&Q\ar[r]^\phi\ar[d]&Q'\ar[d]\\
&0&0
}
\]
After replacing $A'\oplus B'$ by $A'\oplus B'/\ker(\gamma\oplus\id)$ and $B'$ by $B'/\ker(\beta)$ to turn the first two rows into a morphism of short exact sequences, the Snake Lemma shows that $\phi$ is an isomorphism, and hence
\[
\ind(\ol\alpha)=\big|Q'\big|=\big| Q\big|=\ind(\alpha+\beta).
\]
\end{proof}
}}

\subsection{Gluing of the Local Models} \label{gluingofthelocalmodels}
Now, we are ready to prove the main theorem. 
\begin{proof}[\textbf{Proof of Theorem \ref{main theorem}}]

For $\ol{x}$ a geometric point on $\wt{\fM}^{\gl,\ev}_{\red}$, by Section \ref{localmodelsubsection}, there is an \etale neighborhood $U_{\ol{x}}$ of $\delta_{\red}^{\ev}(\ol{x})$ and $U_{\ol{x}}^{\gl}$ such that the Theorem \ref{localsplittingformula} is satisfied.  As $\ol{x}$ goes over the geometric points on $\wt{\fM}^{\gl,\ev}_{\red}$, we obtain an \etale cover $\bigsqcup_{\ol{x}} U_{\ol{x}}$ of $\prod_{i=1}^r \wt{\fM}^{\ev}_{\btau_i, \red}$.

Let $\fV$ be a generic displacement vector defined in Definition \ref{globalgluingdata}. By Proposition \ref{localsplittingformula} and  Remark \ref{compatibilityremark}, there are closed substacks $K_{U_{\ol{x}}}$ of $U_{\ol{x}}\times \PP^1$ that induces the rational equivalence condition \eqref{induceformulacondition}. We first show that for the geometric points $\ol{x}$ and $\ol{x}'$ in the same connected component of $\wt{\fM}^{\gl,\ev}_{\red}$, the indices $\fI_{\ol{x}}$ in \eqref{induceformulacondition} are the same. It is sufficient to show $\fI_{\ol{x}} = \fI_{\ol{x}'}$ supposing $\ol{x}$ is a generization of $\ol{x}'$. As the idealized structure on $\wt{\fM}_{\red}^{\gl,\ev}$ is coherent, by \cite[Prop II.2.6.1]{LogAlgebraicGeometry}, the ideal $L_{\ol{x}}$ is generated by $L_{\ol{x}'}$ under the map $Q_{\ol{x}'}\rightarrow Q_{\ol{x}}$. We then obtain an open embedding of stacks $\cA_{Q_{\ol{x}}, L_{\ol{x}}} \rightarrow \cA_{Q_{\ol{x}'}, L_{\ol{x}'}}$. Then there is a fiber diagram
\begin{equation*} 
\begin{tikzcd}[column sep = small, row sep = small]
\A_{\ol{x}}^{\gl, \ev} \arrow[d,"\delta_{\cA,\ol{x}'}^{\ev}"] \arrow[r,hook]& \A_{\ol{x}'}^{\gl, \ev} \arrow[d,"\delta_{\cA,\ol{x}}^{\ev}"] \arrow[r] & X'_{\btau} \arrow[d]\\
\A_{\ol{x}}^{\ev} \arrow[d] \arrow[r,hook]& \A^{\ev}_{\ol{x}'} \arrow[r] \arrow[d] & \prod_{i=1}^r X'_{\btau_i} \arrow[d] \\
\cA_{Q_{\ol{x}},L_{\ol{x}}}\arrow[r,hook]& \cA_{Q_{\ol{x}'},L_{\ol{x}'}}\arrow[r] & \prod_{i=1}^r \cA_{\btau_i}.
\end{tikzcd}
\end{equation*}
As the indices $\fI_{\ol{x}}$ and $\fI_{\ol{x}'}$ are the degrees of the morphisms $\delta_{\cA,\ol{x}}^{\ev}$ and $\delta_{\cA,\ol{x}'}^{\ev}$, it follows from the diagram that they are the same.

Assume that $\wt{\fM}^{\gl,\ev}_{\red}$ has one connected component and let $\fI$ be the index $\fI_{\ol{x}}$ for any geometric point $\ol{x}$. Let $K_{\btau}$  be the closure of the image of $\bigsqcup K_{U_{\ol{x}}}$ in $\prod_{i=1}^r \wt{\fM}_{\btau_i,\red}^{\ev}\times \PP^1$. Let $\alpha$ be the projection map $\alpha: K_{\btau}\rightarrow \PP^1$. Then, as algebraic cycles in $\prod_{i=1}^r \wt{\fM}_{\btau_i,\red}^{\ev}$,
\begin{equation*}
    \begin{split}
        [\alpha^{-1}(1)] & = \big[\bigcup_{\ol{x}} \psi_{\ol{x}} \circ \delta^{\ev}_U(U_{\ol{x}}^{\gl})\big] = \big[\delta_{\red}^{\ev}(\wt{\fM}_{\red}^{\gl,\ev})\big] \\
        [\alpha^{-1}(0)] & = \sum_{[\brho]\in \Delta(\fV)} \frac{m_{[\brho]}}{\fI} \cdot \big[\bigcup_{\ol{x}} \psi_{\ol{x}}( U_{\ol{x}}^{\brho})\big] \\
        & = \sum_{[\brho]\in \Delta(\fV)} \frac{m_{[\brho]}}{\fI} \cdot \big[\prod_{i=1}^r \wt{j}_{\brho_i, \btau_i}(\wt{\fM}^{\ev}_{\brho_i,\red})],
    \end{split}    
\end{equation*}
where $\psi_{\ol{x}}$ is the \etale map from $U_{\ol{x}}$ to $\prod_{i=1}^r \wt{\fM}_{\btau_i,\red}^{\ev}$ and $\wt{j}_{\brho_i, \btau_i}: \wt{\fM}^{\ev}_{\brho_i,\red} \rightarrow \wt{\fM}^{\ev}_{\btau_i,\red}$ is the finite map induced from the contraction morphism from $\brho_i$ to $\btau_i$ for each $i=1,\ldots,r$. As the degree of $\delta^{\ev}_{\red}$ is the same as $\delta^{\ev}_U$, we obtain the equation
\begin{equation} \label{diagonalequationkey}
    \delta^{\ev}_{\red*}[\wt{\fM}_{\red}^{\gl,\ev}] = \fI \cdot [\delta^{\ev}_{\red}(\wt{\fM}_{\red}^{\gl,\ev})] = \sum_{[\brho] \in \Delta(\fV)} m_{[\brho]} \cdot \prod_{i=1}^r[\wt{j}_{\brho_i, \btau_i}(\wt{\fM}_{\brho_i,\red}^{\ev})].
\end{equation} Following the notations in  Proposition \ref{gluedmodulireducedsetup}, we have $\gamma_i$ from $\wt{\fM}^{\ev}_{\btau_i}$ to $\fM^{\ev}_{\btau_i}$ and $\beta_i$ from the reduced induced stack $\wt{\fM}^{\ev}_{\btau_i,\red}$ to $\wt{\fM}^{\ev}_{\btau_i}$, and 
\begin{equation*}
\begin{split}
    \delta'_*[\fM^{\ev}_{\btau}] & = (\prod_{i=1}^r\gamma_i \circ \beta_i)_* \delta^{\ev}_{\red*}[\wt{\fM}^{\gl,\ev}_{\red}]\\
    & \overset{\eqref{diagonalequationkey}}{=} \sum_{[\brho] \in \Delta(\fV)} m_{[\brho]} \cdot \prod_{i=1}^r(\gamma_i\circ \beta_{i})_*[\wt{j}_{\brho_i, \btau_i}(\wt{\fM}_{\brho_i,\red}^{\ev})]\\
    & \overset{(1)}{=} \sum_{[\brho] \in \Delta(\fV)} m_{[\brho]} \cdot \prod_{i=1}^r[j_{\brho_i,\btau_i}(\fM^{\ev}_{\brho_i})]\\
    & = \sum_{[\brho] \in \Delta(\fV)} \frac{m_{[\brho]}}{|\Aut(\brho_i/\btau_i)|} \cdot \prod_{i=1}^rj_{\brho_i,\btau_i*}[\fM^{\ev}_{\brho_i}],
\end{split}
\end{equation*}
with $j_{\brho_i, \btau_i}: {\fM}^{\ev}_{\brho_i} \rightarrow {\fM}^{\ev}_{\btau_i}$ from the contraction morphism $\brho_i\rightarrow \btau_i$. Since $\fM_{\brho_i}^{\ev}$ is reduced over $B$ as shown in \cite[Prop 3.28]{ACGSPunc}, 
\begin{equation*}
    j_{\brho_i,\btau_i}(\fM_{\brho_i}^{\ev}) = \gamma_i\circ \beta_i \circ \wt{j}_{\brho_i,\btau_i}(\wt{\fM}_{\brho_i,\red}^{\ev}).
\end{equation*}
By \cite[Prop 5.5]{ACGSPunc}, $\gamma_i$ induces an isomorphism on reductions. Hence $\gamma_i\circ \beta_i$ has degree one and we get equality $(1)$. The last equality follows from the fact that $j_{\brho_i,\btau_i}$ is finite of degree $|\Aut(\brho_i/\btau_i)|$.

Now suppose $\wt{\fM}^{\gl,\ev}_{\red}$ has more than one component, then  by Lemma \ref{gluingreducedlemma} and Proposition \ref{isoafterenhanced}, moduli spaces $\wt{\fM}_{\btau}^{\ev}$ and $\fM^{\ev}_{\btau}$ have more than one component. For each component, we have equation \eqref{diagonalequationkey}. As $m_{[\brho]}$ is independent of the geometric point $\ol{x}$ from the component, the above equation holds for general $\fM_{\btau}^{\ev}$. Then, following Theorem \ref{setuptheorem}, we finish the proof of Theorem \ref{main theorem}.
\end{proof}

\appendix

\section{Lift of Artin Cones} \label{Appendix}

Recall that for an algebraic group $G$ and a scheme $X$ with a $G$-action, the quotient stack $[X/G]$ is the groupoid fibered over the category of schemes, such that
\begin{enumerate}
\item An object over a scheme $B$ is a diagram 
 $ \Bigg\{
    \begin{tikzcd}[column sep = small]
    E \arrow{r}{h}  \arrow{d}{\pi} & X \\
    B,
    \end{tikzcd}
    \Bigg\}$
where $E$ is a principal $G$-bundle over $B$ and $h: E\rightarrow X$ is a $G$-equivariant map. 
\item A morphism from an object  $ \Bigg\{
    \begin{tikzcd}[column sep = small,row sep = small]
    E' \arrow{r}{h'}  \arrow{d}{\pi} & X \\
    B',
    \end{tikzcd}
    \Bigg\}$ over $B'$ to an object  $ \Bigg\{
    \begin{tikzcd}[column sep = small,row sep = small]
    E \arrow{r}{h}  \arrow{d}{\pi} & X \\
    B,
    \end{tikzcd}
    \Bigg\}$ over $B$ is a pair $(g,g')$ with $g:B'\rightarrow B$ and $g': E'\rightarrow E\times_B B'$ such that $g'$ is a $G$-equivariant isomorphism and $h' = h\circ g'$.
\end{enumerate} 

Let us first show a general lemma regarding quotient stacks.

\begin{lemma} \label{quotientstacklemma}
Let $G$ be an algebraic group and $X$ be a scheme with a $G$-action $\gamma: G\times X\rightarrow X$. Let $Y$  and $Z$ be algebraic stacks.

Assume there is a $G$-invariant morphism $f':X\rightarrow Y$, hence a map $f$ from the quotient stack $[X/G]$ to $Y$. Let $g:Z\rightarrow Y$ be a representable morphism of algebraic stacks. There is a $G$-action on scheme $X\times_Y Z$ induced by its action on $X$. Then, there is a $2$-isomorphism
\begin{equation*}
[X\times_Y Z/G] \rightarrow [X/G]\times_{Y} Z.
\end{equation*}
\end{lemma}
\begin{proof}
By \cite[Lemma 2.3.2]{wang2011moduli}, there is a $2$-Cartesian diagram
\begin{equation*}
\begin{tikzcd}[column sep = small, row sep = small]
{[}X\times_Y Z/G{]} \arrow[r] \arrow[d] & {[}Z/G{]} \arrow[d] \\
{[}X/G{]} \arrow[r] & {[}Y/G{]},
\end{tikzcd}
\end{equation*}
where $[Z/G]$ and $[Y/G]$ are the quotient stacks induced by the trivial $G$-actions. As $[Z/G]$ is $2$-isomorphic to the product stack $Z\times BG$ and $[Y/G]$ is $2$-isomorphic to $Y\times BG$, we obtain that 
\begin{equation*}
[Z/G] = [Y/G]\times_Y Z.
\end{equation*}
Hence
\begin{equation*}
[X\times_Y Z/G] = [X/G]\times_{[Y/G]} [Z/G] = [X/G] \times_Y Z.
\end{equation*}
\end{proof}

Let $P,Q$ be toric monoids. For a monoid morphism $m:P \rightarrow Q$, we let $f:S_Q \rightarrow S_P$ be the associated toric morphism, $f_T:T_Q\rightarrow T_P$ be the algebraic torus morphism and $f_\A: \A_Q \rightarrow \A_P$ the morphism of Artin cones. 

{\color{black}
\begin{proposition} \label{cartesian}
Let ${[}S_Q \times T_P/T_Q{]}$ be the toric stack obtained by the torus action
\begin{equation*}
T_Q\times (S_Q\times T_P) \rightarrow S_Q\times T_P
\end{equation*}
associated to the monoid morphism
\begin{equation*}
\begin{split}
 Q\oplus P^{\gp}\rightarrow Q^{\gp}\oplus Q\oplus P^{\gp}, \quad (q,p)\mapsto (q+m^{\gp}(p), q, p).
\end{split}
\end{equation*}

Then there is a Cartesian diagram of Artin stacks
\begin{equation} \label{bigcomm}
    \begin{tikzcd}
    {[}S_Q \times T_P/T_Q{]} \arrow{r}{g}
    \arrow{d}{\eta} & S_P \arrow{d}{\chi}
    \\
    \A_Q \arrow{r}{f_\A} & \A_P,
    \end{tikzcd}
\end{equation}
where $\chi:S_P\rightarrow \A_P$ is the quotient map, the morphism $\eta$ is induced from the $T_Q$-equivariant map $S_Q\times T_P$ to $S_Q$ by taking the projection and $g$ is induced from the $\wt{\eta}: S_Q\times T_P \rightarrow S_P$ associated to the monoid morphism 
\begin{equation} \label{monoidmap}
\begin{split}
P \rightarrow Q\oplus P^{\gp}, \quad p\mapsto (m(p),-p),
\end{split}
\end{equation}
invariant under the $T_Q$-action on $S_Q\times T_P$.
\end{proposition}
}
\begin{proof}
With trivial $T_Q$-action on $\cA_P$, the map $S_Q\xrightarrow{f} S_P \rightarrow \cA_P$ is $T_Q$-invariant. Since $\chi:S_P\rightarrow \cA_P$ is representable, by Lemma \ref{quotientstacklemma},
\begin{equation*}
\begin{split}
\cA_Q\times_{\cA_P} S_P  = [S_Q\times_{\cA_P} S_P/T_Q],
\end{split}
\end{equation*}
with $T_Q$ acting on the fiber product by acting on $S_Q$. On the other hand, there is a commutative diagram
\begin{equation*}
    \begin{tikzcd}
      S_Q\times T_P \arrow[r,"f\times \id"] \arrow[d,"\pr_1"]&
      S_P\times T_P \arrow[r,"a"] \arrow[d,"\pr_1"]& S_P \arrow[d,"\chi"] \\
      S_Q \arrow[r,"f"] & S_P \arrow[r,"\chi"] & \cA_P,
    \end{tikzcd}
\end{equation*}
where $a$ is the group action. By \cite[\href{https://stacks.math.columbia.edu/tag/04M9}{Tag 04M9}]{stacks-project}, the commutative diagram on the right hand side is $2$-Cartesian. We then obtain that $S_Q\times_{\cA_P} S_P = S_Q\times T_P$ as both squares are fiber diagrams. The induced $T_Q$-action on $S_Q\times T_P$ is given by the monoid morphism in \eqref{monoidmap}, which is the unique action that makes the above fiber diagram $T_Q$-equivariant.
\end{proof}

Proposition \ref{cartesian} can be generalized to the toric strata of affine toric varieties.

\begin{corollary}\label{IdealFanLift} 
Let $m:P\rightarrow Q$ be a morphism of toric monoids. Let $K$ be an ideal of $P$ and $L$ be an ideal of $Q$ such that $m(K) \subseteq L$. Then there is a Cartesian diagram of idealized log stacks
\begin{equation*}
    \begin{tikzcd}
      {[}S_{Q,L} \times T_P/T_Q{]} \arrow{r}{g} \arrow{d}{\eta} & S_{P,K} \arrow{d}{\chi}\\
      \A_{Q,L} \arrow{r}{f_m} & \A_{P,K},
    \end{tikzcd}
\end{equation*}
where $\chi$ is the canonical quotient map and $f_m$ is the map of toric stacks associated to the monoid morphism $m$. The morphism $\eta$ is induced from the $T_Q$-equivariant map $S_{Q,L}\times T_P$ to $S_{Q,L}$ by taking the projection and $g$ is induced from the $\wt{\eta}: S_{Q,L}\times T_P \rightarrow S_{P,K}$ associated to the monoid morphism 
\begin{equation*}
\begin{split}
P \rightarrow Q\oplus P^{\gp}, \quad p\mapsto (m(p),-p),
\end{split}
\end{equation*}
invariant under the $T_Q$-action on $S_{Q,L}\times T_P$.

\end{corollary}
\begin{proof}
Since $S_{P,K} = \A_{P,K} \times_{\A_P} S_P$, we obtain that 
\begin{equation*}
\begin{split}
\A_{Q,L} \times_{\A_{P,K}} S_{P,K} & = \A_{Q,L} \times_{\A_{P,K}} (\A_{P,K} \times_{\A_P} S_P)\\
& = \A_{Q,L} \times_{\A_P} S_P \\
& =  \A_{Q,L} \times_{\cA_Q} (\cA_Q\times_{\cA_P} S_P)\\
& = \A_{Q,L} \times_{\A_Q} [S_Q \times T_P/T_Q].
\end{split}
\end{equation*}
The map $[S_Q\times T_P/T_Q]\rightarrow \A_Q$ is induced by the $T_Q$-invariant map $S_Q\times T_P\xrightarrow{\operatorname{pr_1}} S_Q \rightarrow \cA_Q$. By Lemma \ref{quotientstacklemma}, we then obtain that
\begin{equation*}
     \begin{split}
        \A_{Q,L} \times_{\A_Q} [S_Q \times T_P/T_Q]& = [\cA_{Q,L}\times_{\cA_Q} (S_Q\times T_P)/T_Q]\\
         & = [(\cA_{Q,L}\times_{\cA_Q}S_Q)\times T_P /T_Q] = [S_{Q,L}\times T_P/T_Q].
         \end{split}
\end{equation*}
\end{proof}

\section{Logarithmic Fiber Product of Toric Varieties}
\label{AppendixB}
In this section, we study the logarithmic fine, saturated fiber products of toric varieties. Unlike the fiber product in the category of schemes, the log fine, saturated fiber products of toric varieties are totally determined by the fiber product of the fans.

\begin{definition} \label{fiberproductoffans}
Let $\Sigma(X)\rightarrow \Sigma(Y)$ and $\Sigma(Z) \rightarrow \Sigma(Y)$ be morphisms of fans. Define the \textit{fiber product of fans} $\Sigma(X)\times_{\Sigma(Y)} \Sigma(Z)$ to be the fan $(\wt{\Sigma},\wt{N})$ with
\begin{enumerate}
\item $\wt{N}$ being the fiber product of lattices $N(X)\times_{N(Y)} N(Z)$,
\item $\wt{\Sigma}$ consisting of the cones $\sigma_X\times_{\sigma_Y} \sigma_Z$ with $\sigma_X \in \Sigma(X)$, $\sigma_Y \in \Sigma(Y)$ and $\sigma_Z \in \Sigma(Z)$.
\end{enumerate}
\end{definition}

\begin{lemma} \label{molcholemma}
Let $f: X\rightarrow Y$ and $g: Z\rightarrow Y$ be toric morphisms of toric varieties. Then, the fine, saturated fiber product $X\times_Y^{\fs} Z$ is a disjoint union of toric varieties and each of them is isomorphic to the toric variety of the fiber product of fans $\Sigma(X)\times_{\Sigma(Y)} \Sigma(Z)$. The number of the components is the lattice index $[(\im \alpha)^{\sat}: \im \alpha]$ under the map
\begin{equation} \label{alphadefinition}
\alpha: N(X)\times N(Z) \rightarrow N(Y), \quad (x,z)\mapsto f_N(x)-g_N(z)
\end{equation}
with $f_N$ and $g_N$ the lattice maps associated to the toric morphisms $f$ and $g$.

In particular, if the lattice map $\alpha$ has full dimensional image in $N(Y)$, then the lattice index is $[N(Y):\im \alpha]$.
\end{lemma}

\begin{proof}
The fine, saturated logarithmic fiber product of the toric varieties is discussed in \cite[Rmk 2.2.5]{MolchoSam2021Ussr} in detail. By \cite{MolchoSam2021Ussr}, the fine, saturated log fiber product $X\times_Y^{\fs} Z$ is a disjoint union of schemes, each of which is isomorphic to the toric variety of the fiber product of fans $\Sigma(X)\times_{\Sigma(Y)} \Sigma(Z)$. The reason of the fiber products containing several components is due to the fact that the torsion subgroup $\Tor$ of the fibered sum of monoids $M(X)\oplus_{M(Y)} M(Z)$ is nontrivial, with $M(X)$, $M(Y)$ and $M(Z)$ the character lattices. As we are working over a field of characteristic $0$, the number of the components is the order of the group $\Tor$.

Note that we have an exact sequence of monoids
\begin{equation} \label{dualmonoidmap}
\begin{split}
 M(Y) \xrightarrow{\phi} M(X)\oplus M(Z) \xrightarrow{\psi'} M(X)\oplus_{M(Y)} M(Z) \rightarrow 0,
\end{split}
\end{equation}
with $\phi = (f^*_N,-g^*_N)$ and $\psi'$ the cokernel map of $\phi$.
Let 
\begin{equation*}
\psi: M(X)\oplus M(Z)\xrightarrow{\psi'} M(X)\oplus_{M(Y)} M(Z) \xrightarrow{q} M(X)\oplus_{M(Y)} M(Z)/\Tor.
\end{equation*} Then, the torsion group
\begin{equation*}
\begin{split}
\Tor = \ker q = \psi'(\ker \psi) \cong
\ker \psi / (\ker \psi \cap \ker \psi') = \ker \psi / \im \phi.
\end{split}
\end{equation*}
As $\im \phi$ is a full dimensional sublattice of $\ker \psi$, by lattice geometry, the order of the torsion group $\Tor$ is the same as the lattice index $[(\im \phi)^*: (\ker \psi)^*]$.

We finish the proof by showing that under the inclusion of the lattice 
\begin{equation*}
\alpha': (\im \phi)^* \rightarrow N(Y),
\end{equation*} the image of $(\im\phi)^*$ is $(\im \alpha)^{\sat}$, and the image of $(\ker\psi)^* \rightarrow (\im \phi)^*$ is $(\im \alpha)$, with $\alpha$ defined in \eqref{alphadefinition}. 

The cokernel of the map $\alpha'$ is isomorphic to $(\ker\phi)^*$, following the dual of the short exact sequence
\begin{equation*}
0 \rightarrow \ker \phi \rightarrow M(Y) \rightarrow \im \phi \rightarrow 0.
\end{equation*}
Hence, $\coker \alpha'$ is torsion free. Therefore, the image of $(\im \phi)^*$ under $\alpha'$ is a saturated sublattice of $N(Y)$ containing $(\im \alpha)$.

Note that $\alpha$ is the dual map of $\phi$ in \eqref{dualmonoidmap}. Taking the dual of the exact sequence \eqref{dualmonoidmap}, we obtain an exact sequence of the lattices
\begin{equation*}
0 \rightarrow N(X)\times_{N(Y)} N(Z) \xrightarrow{\psi^*} N(X)\times N(Z) \xrightarrow{\alpha} N(Y).
\end{equation*}
Therefore $\im \alpha \cong \coker \psi^* = (\ker \psi)^*$, with the second equality following the dual of the short exact sequence
\begin{equation*}
0\rightarrow \ker(\psi) \rightarrow M(Y) \xrightarrow{\psi} \im(\psi) \rightarrow 0.
\end{equation*}
The lattice quotient 
\begin{equation*}
[(\im\phi)^*: (\ker \psi)^*] = [(\im \alpha)^{\sat}: \im \alpha],
\end{equation*}
which equals to the order of the torsion group $\Tor$.

\end{proof}




\printbibliography[title = {Reference}]

\end{document}